\documentclass[12pt,letter]{amsart}

\usepackage{amssymb,latexsym, amsmath, amsxtra, bbm}
\usepackage[dvips]{graphics}
\usepackage{xypic}
\usepackage{verbatim}
\usepackage[abs]{overpic}
\usepackage{hyperref}

\allowdisplaybreaks[3]

\theoremstyle{plain}
        \newtheorem{theorem}{Theorem}[section]
        \newtheorem*{theorem*}{Theorem}
        \newtheorem*{conj*}{Conjecture}
        \newtheorem{lemma}[theorem]{Lemma}
        \newtheorem{prop}[theorem]{Proposition}
        \newtheorem{cor}[theorem]{Corollary}

\theoremstyle{definition}
        \newtheorem{definition}[theorem]{Definition}
        \newtheorem{rem}[theorem]{Remark}

         \newtheorem*{assumptions}{The Assumptions}

\theoremstyle{remark}
        \newtheorem*{remark}{Remark}

        \newtheorem{question}{Question}

\numberwithin{equation}{section}
\numberwithin{theorem}{section}
\numberwithin{table}{section}
\numberwithin{figure}{section}

\providecommand{\defn}[1]{\emph{#1}}

\newcommand{\diam}  {\operatorname{diam}}

\newcommand{\inte}  {\operatorname{inte}}
\newcommand{\inter}  {\operatorname{int}}

\newcommand{\card} {\operatorname{card}}
\newcommand{\supp}{\operatorname{supp}}
\newcommand{\R}{\mathbb{R}}      

\newcommand{\C}{\mathbb{C}}      

\newcommand{\N}{\mathbb{N}}      
\newcommand{\Z}{\mathbb{Z}}      
\newcommand{\T}{\mathbb{T}}      

\providecommand{\Abs}[1]{\left|#1\right|}

\providecommand{\Norm}[1]{\left\|#1\right\|}
\renewcommand{\:}{\colon}

\def\({\left(}
\def\){\right)}
\def\[{\left[}
\def\]{\right]}

\newcommand{\crit}{\operatorname{crit}}
\newcommand{\post}{\operatorname{post}}

\newcommand{\CC}{\mathcal{C}}

 \newcommand{\DD}{\mathbf{D}}

\newcommand{\PP}{\mathbf{P}}

\newcommand{\X} {\mathbf{X}}

\newcommand{\E} {\mathbf{E}}

\newcommand{\V} {\mathbf{V}}

\newcommand{\W} {\mathbf{W}}

\newcommand{\G}{\mathcal{G}}

\renewcommand{\T}{\mathcal{T}}

\newcommand{\VV}{\mathcal{V}}

\newcommand{\EE}{\mathcal{E}}

\newcommand{\FF}{\mathfrak{F}}

\newcommand{\GG}{\mathfrak{G}}

\newcommand{\UU}{\mathfrak{U}}

\newcommand{\GGG}{G}

\newcommand{\CCC}{C}

\newcommand{\PPP}{\mathcal{P}}

\newcommand{\MMM}{\mathcal{M}}

\newcommand{\Holder}[1] {\CCC^{0,#1}}

\newcommand{\RR}{\mathcal{L}}

\begin{document}
\title[Weak Expansion Properties and Large Deviation Principles]{Weak Expansion Properties and Large Deviation Principles for Expanding Thurston Maps}
\author{Zhiqiang~Li}
\address{Department of Mathematics, UCLA, Los Angeles CA 90095-1555}
\thanks{The author was partially supported by NSF grants DMS-1162471 and DMS-1344959.}
\email{lizq@math.ucla.edu}




\begin{abstract}
In this paper, we prove that an expanding Thurston map $f\: S^2 \rightarrow S^2$ is asymptotically $h$-expansive if and only if it has no periodic critical points, and that no expanding Thurston map is $h$-expansive. As a consequence, for each expanding Thurston map without periodic critical points and each real-valued continuous potential on $S^2$, there exists at least one equilibrium state. For such maps, we also establish  large deviation principles for iterated preimages and periodic points. It follows that iterated preimages and periodic points are equidistributed with respect to the unique equilibrium state for an expanding Thurston map without periodic critical points and a potential that is H\"older continuous with respect to a visual metric on $S^2$. 
\end{abstract}

\maketitle

\tableofcontents

\section{Introduction}

The theory of discrete-time dynamical systems studies qualitative and quantitative properties of orbits of points in a space under iterations of a given map. Various conditions can be imposed upon the map to simplify the orbit structures, which in turn lead to results about the dynamical system under consideration. One such well-known condition is expansiveness. Roughly speaking, a map is expansive if no two distinct orbits stay close forever. Expansiveness plays an important role in the exploitation of hyperbolicity in smooth dynamical systems, and in complex dynamics in particular (see for example, \cite{Ma87} and \cite{PU10}).

In the context of continuous maps on compact metric spaces, there are two weaker notions of expansion, called \emph{$h$-expansiveness} and \emph{asymptotic $h$-expansiveness}, introduced by R.~Bowen \cite{Bow72} and M.~Misiurewicz \cite{Mi73}, respectively. Forward-expansiveness implies $h$-expan\-siveness, which in turn implies asymptotic $h$-expansiveness \cite{Mi76}. Both of these weak notions of expansion play important roles in the study of smooth dynamical systems (see \cite{Bu11, DFPV12, DM09, DN05, LVY13}). Moreover, any smooth map on a compact Riemannian manifold is as\-ymp\-totically $h$-expansive \cite{Bu97}. Recently, N.-P.~Chung and G.~Zhang extended these concepts to the context of a continuous action of a countable discrete sofic group on a compact metric space \cite{CZ14}.

The dynamical systems that we study in this paper are induced by \emph{expanding Thurston maps}, which are a priori not differentiable. Thurston maps are branched covering maps on the sphere $S^2$ that generalize rational maps with finitely many postcritical points on the Riemann sphere. More precisely, a (non-homeomorphic) branched covering map $f\: S^2\rightarrow S^2$ is a Thurston map if it has finitely many critical points each of which is preperiodic. These maps arose in W.~P.~Thurston's characterization of postcritically-finite rational maps (see \cite{DH93}). See Section~\ref{sctThurstonMap} for a more detailed introduction to Thurston maps.

Inspired by the analogy to Cannon's conjecture in geometric group theory (see for example, \cite[Section~5 and Section~6]{Bon06}), M.~Bonk and D.~Meyer investigated extensively properties of expanding Thurston maps \cite{BM10}. (For a precise formulation of these analogies via the so-called Sullivan's dictionary, see \cite[Section~1]{HP09}.) Such maps share many common features of rational maps. For example, for each such map $f$, there are exactly $1+\deg f$ fixed points, counted with a natural weight induced by the local degree at each point, where $\deg f$ denotes the topological degree of the map $f$ \cite{Li13}; there exists a unique measure of maximal entropy (see for example, \cite{BM10}), with respect to which iterated preimages and periodic points  are equidistributed in some appropriate sense (see for example, \cite{Li13}). More generally, for each \emph{potential} $\phi \: S^2 \rightarrow \R$ that is H\"older continuous with respect to some natural metric induced by $f$, there exists a unique equilibrium state, with respect to which iterated preimages are equidistributed \cite{Li14}. 

P.~Ha\"issinsky and K.~Pilgrim investigated branched covering maps in a more general context \cite{HP09}. We will focus on expanding Thurston maps in this paper.

Let $(X,d)$ be a compact metric space, and $g\: X\rightarrow X$ a continuous map on $X$. Denote, for $\epsilon>0$ and $x\in X$,
$$
\Phi_\epsilon (x) = \{ y\in X \,|\, d(g^n(x),g^n(y)) \leq \epsilon \mbox{ for all } n \geq 0 \}.
$$
The map $g$ is called \defn{forward expansive} if there exists $\epsilon >0$ such that $\Phi_\epsilon(x)=\{x\}$ for all $x\in X$. By R.~Bowen's definition in \cite{Bow72}, the map $g$ is \defn{$h$-expansive} if there exists $\epsilon>0$ such that the topological entropy $h_{\operatorname{top}}(g|_{\Phi_\epsilon(x)})=h_{\operatorname{top}}(g,\Phi_\epsilon(x))$ of $g$ restricted to $\Phi_\epsilon(x)$ is $0$ for all $x\in X$. One can also formulate asymptotic $h$-expansiveness in a similar spirit, see for example, \cite[Section~2]{Mi76}. However, in this paper, we will adopt equivalent formulations from \cite{Do11}. See Section~\ref{subsctAsymHExpConcepts} for details.

Another way to formulate forward expansiveness is via distance expansion. We say that $g\: X\rightarrow X$ is \defn{distance-expanding} (with respect to the metric $d$) if there exist constants $\lambda>1$, $\eta>0$, and $n\in\N$ such that for all $x,y\in X$ with $d(x,y)\leq \eta$, we have $d(g^n(x),g^n(y))\geq \lambda d(x,y)$. If $g$ is forward expansive, then there exists a metric $\rho$ on $X$ such that the metrics $d$ and $\rho$ induce the same topology on $X$ and $g$ is distance-expanding with respect to $\rho$ (see for example, \cite[Theorem~4.6.1]{PU10}). Conversely, if $g$ is distance-expanding, then it is forward expansive (see for example, \cite[Theorem~4.1.1]{PU10}). So roughly speaking, if $g$ is forward expansive, then the distance between two points that are close enough grows exponentially under forward iterations of $g$.

Since a Thurston map, by definition, has to be a branched covering map, we can always find two distinct points that are arbitrarily close to a critical point (thus arbitrarily close to each other) and that are mapped to the same point. Thus a Thurston map cannot be forward expansive. In order to impose some expansion condition, it is then natural to consider backward orbits. We say that a Thurston map is expanding if for any two points $x,y\in S^2$, their preimages under iterations of the map gets closer and closer. See Definition~\ref{defExpanding} for a precise formulation.

The expansion property of expanding Thurston maps seems to be rather strong. However, as a part of our first main theorem below, we will show that no expanding Thurston map is $h$-expansive.

\begin{theorem}  \label{thmWeakExpansion}
Let $f\: S^2 \rightarrow S^2$ be an expanding Thurston map. Then $f$ is asymptotically $h$-expansive if and only if $f$ has no periodic critical points. Moreover, $f$ is not $h$-expansive.
\end{theorem}

When R.~Bowen introduced $h$-expansiveness in \cite{Bow72}, he mentioned that no diffeomorphism of a compact manifold was known to be not $h$-expansive. M.~Misiurewicz then produced an example of a diffeomorphism that is not asymptotically $h$-expansive \cite{Mi73}. M.~Lyubich showed that any rational map is asymptotically $h$-expansive \cite{Ly83}. J.~Buzzi established asymptotic $h$-expan\-sive\-ness of any $\CCC^\infty$-map on a compact Riemannian manifold \cite{Bu97}. Examples of $\CCC^\infty$-maps that are not $h$-expansive were given by M.~J.~Pacifico and J.~L.~Vieitez \cite{PV08}. Our Theorem~\ref{thmWeakExpansion} implies that any rational expanding Thurston map (i.e., any postcritically-finite rational map whose Julia set is the whole sphere (see \cite[Proposition~19.1]{BM10})) is not $h$-expansive. 

Expanding Thurston maps may be the first example of a class of a priori non-differentiable maps that are not $h$-expansive but may be asymptotically $h$-expansive depending on the property of orbits of critical points.

As an immediate consequence of Theorem~\ref{thmWeakExpansion} and the result of J.~Buzzi \cite{Bu97} mentioned above, we get the following corollary, which partially answers a question of K.~Pilgrim (see Problem~2 in \cite[Section~21]{BM10}).

\begin{cor}    \label{corNotConjDiff}
An expanding Thurston map with at least one periodic critical point cannot be conjugate to a $\CCC^\infty$-map on the Euclidean $2$-sphere.
\end{cor}

\smallskip

Our real motivation to investigate Theorem~\ref{thmWeakExpansion} comes from another basic theme in the study of dynamical systems, namely, the investigation of the measure-theoretic entropy and measure-theoretic pressure, and their maximizing measures known as the measures of maximal entropy and equilibrium states, respectively.

For a continuous map on a compact metric space, we can consider the topological pressure as a weighted version of the topological entropy, with the weight induced by a real-valued continuous function, called \emph{potential}. The Variational Principle identifies the topological pressure with the supremum of its measure-theoretic counterpart, the measure-theoretic pressure, over all invariant Borel probability measures \cite{Bow75, Wa76}. Under additional regularity assumptions on the map and the potential, one gets existence and uniqueness of an invariant Borel probability measure maximizing measure-theoretic pressure, called the equilibrium state for the given map and the potential. When the potential is $0$, the corresponding equilibrium state is known as the measure of maximal entropy. Often periodic points and iterated preimages are equidistributed in some appropriate sense with respect to such measures. See Section~\ref{subsctThermDynForm} for concepts mentioned here.

The existence, uniqueness, and various properties of equilibrium states have been studied in many different contexts (see for example, \cite{Bow75, Ru89, Pr90, KH95, Zi96, MauU03, BS03, Ol03, Yu03, PU10, MayU10}).

M.~Misiurewicz showed that asymptotic $h$-expansiveness guarantees that the measure-theoretic entropy $\mu\mapsto h_\mu(f)$ is upper semi-continuous \cite{Mi76}. We then get the following corollary from Theorem~\ref{thmWeakExpansion}.

\begin{cor}    \label{corUSC}
Let $f\: S^2 \rightarrow S^2$ be an expanding Thurston map without periodic critical points. Then the measure-theoretic entropy $h_\mu(f)$ considered as a function of $\mu$ on the space $\MMM(S^2,f)$ of $f$-invariant Borel probability measures is upper semi-continuous. Here $\MMM(S^2,f)$ is equipped with the weak$^*$ topology.
\end{cor}

Recall that if $X$ is a metric space, a function $h\: X\rightarrow [-\infty,+\infty]$ is \defn{upper semi-continuous} if $\limsup_{y\to x} h(y) \leq h(x)$ for all $x\in \X$.

In \cite{Li14}, we established the existence and uniqueness of the equilibrium state for an expanding Thurston map and a given real-valued H\"older continuous potential. Here the sphere $S^2$ is equipped with a natural metric induced by $f$, called a visual metric. (See Theorem~\ref{thmExistenceUniquenessES}.) The tools we used in \cite{Li14} are from the \emph{thermodynamical formalism}. Neither Theorem~\ref{thmWeakExpansion} nor Corollary~\ref{corUSC} was used there. Note that Corollary~\ref{corUSC} implies a partially stronger existence result than the one obtained in \cite{Li14}.

\begin{theorem}    \label{thmExistenceES}
Let $f\:S^2 \rightarrow S^2$ be an expanding Thurston map without periodic critical points and $\psi \in \CCC(S^2)$ be a real-valued continuous function on $S^2$ (with respect to the standard topology). Then there exists at least one equilibrium state for the map $f$ and the potential $\psi$.
\end{theorem}

See Section~\ref{subsctThermDynForm} for a quick proof after necessary definitions are given precisely.

Once we know the existence and uniqueness of the equilibrium states, one natural question to ask is how periodic points and iterated preimages are distributed with respect to such measures. We know that for an expanding Thurston map, iterated preimages and preperiodic points (and in particular, periodic points) are equidistributated with respect to the unique measure of maximal entropy (see \cite{Li13} and \cite{HP09}).

Some versions of equidistribution of iterated preimages with respect to the unique equilibrium state for an expanding Thurston map and a H\"older continuous potential were obtained in \cite{Li14}. We record them in Proposition~\ref{propWeakConvPreImgWithWeight}. However, similar results for periodic points were inaccessible by the methods used in \cite{Li14} due to technical difficulties arising from the existence of critical points.

In this paper, thanks to Theorem~\ref{thmWeakExpansion} and Corollary~\ref{corUSC}, rather than trying to establish the equidistribution of periodic points directly, we derive some stronger results using a general framework devised by Y.~Kifer \cite{Ki90}. More precisely, we obtain \emph{level-2 large deviation principles} for periodic points with respect to equilibrium states in the context of expanding Thurston maps without periodic critical points and H\"older continuous potentials. We use a variant of Y.~Kifer's result formulated by H.~Comman and J.~Rivera-Letelier \cite{CRL11}, which is recorded in Theorem~\ref{thmAbsLargeDeviationPrincipleCRL} for the convenience of the reader. For related results on large deviation principles in the context of rational maps on the Riemann sphere under additional assumptions, see \cite{PSh96, PSr07, XF07, PRL11, Com09, CRL11}.

Denote the space of Borel probability measures on a compact metric space $X$ equipped with the weak$^*$ topology by $\PPP(X)$. A sequence $\{\Omega_n \}_{n\in\N}$ of Borel probability measures on $\PPP(X)$ is said to satisfy a \defn{level-2 large deviation principle with rate function $I$} if for each closed subset $\FF$ of $\PPP(X)$ and each open subset $\GG$ of $\PPP(X)$ we have 
$$
\limsup\limits_{n\to+\infty} \frac1n \log \Omega_n(\FF) \leq - \inf\{I(x) \,|\, x\in\FF \},
$$
and
$$
\liminf\limits_{n\to+\infty} \frac1n \log \Omega_n(\GG) \geq - \inf\{I(x) \,|\, x\in\GG \}.
$$
We refer the reader to \cite[Section~2.5]{CRL11} and the references therein for a more systematic introduction to the theory of large deviation principles.

In order to apply Theorem~\ref{thmAbsLargeDeviationPrincipleCRL}, we just need to verify three conditions:
\begin{enumerate}
\smallskip
\item[(1)] The existence and uniqueness of the equilibrium state.

\smallskip
\item[(2)] Some characterization of the topological pressure (see Proposition~\ref{propPressureLimitPeriodicPts} and Proposition~\ref{propPressureLimitPreImgs}).

\smallskip
\item[(3)] The upper semi-continuity of the measure-theoretic entropy.
\end{enumerate}

The first condition is established in \cite{Li14}. The second condition is weaker than the equidistribution results, and is within reach. The last condition is known for expanding Thurston maps without periodic critical points by Corollary~\ref{corUSC}. Thus we get the following level-2 large deviation principles.

\begin{theorem}  \label{thmLDP}
Let $f\: S^2\rightarrow S^2$ be an expanding Thurston map with no periodic critical points, and $d$ a visual metric on $S^2$ for $f$. Let $\PPP(S^2)$ denote the space of Borel probability measures on $S^2$ equipped with the weak$^*$ topology. Let $\phi$ be a real-valued H\"older continuous function on $(S^2,d)$, and $\mu_\phi$ be the unique equilibrium state for the map $f$ and the potential $\phi$.

For each $n\in\N$, let $W_n\: S^2\rightarrow\PPP(S^2)$ be the continuous function defined by
$$
W_n(x) = \frac1n \sum\limits_{i=0}^{n-1} \delta_{f^i(x)},
$$
and denote $S_n\phi(x) = \sum\limits_{i=0}^{n-1} \phi\(f^i(x)\)$ for $x\in S^2$. Fix an arbitrary sequence of functions $\{w_n\: S^2\rightarrow \R\}_{n\in\N}$ satisfying $w_n(x)\in [1,\deg_{f^n}(x)]$ for each $n\in\N$ and each $x\in S^2$. We consider the following sequences of Borel probability measures on $\PPP(S^2)$:

{\bf Iterated preimages:} Given a sequence $\{x_n\}_{n\in\N}$ of points in $S^2$, for each $n\in\N$, put
$$
\Omega_n (x_n)  = \sum\limits_{y\in f^{-n}(x_n)} \frac{ w_n(y) \exp(S_n\phi(y))}{\sum_{z\in f^{-n}(x_n)} w_n(z) \exp(S_n\phi(z)) }\delta_{W_n(y)}. 
$$

{\bf Periodic points:} For each $n\in\N$, put
$$
\Omega_n = \sum\limits_{x=f^n(x)} \frac{ w_n(x) \exp(S_n\phi(x))}{\sum_{y=f^n(y)} w_n(y) \exp(S_n\phi(y)) }\delta_{W_n(x)}.
$$

Then each of the sequences $\{\Omega_n(x_n)\}_{n\in\N}$ and $\{\Omega_n\}_{n\in\N}$ converges to $\delta_{\mu_\phi}$ in the weak$^*$ topology, and satisfies a large deviation principle with rate function $I^\phi \: \PPP(S^2) \rightarrow [0,+\infty]$ given by
\begin{equation}    \label{eqRateFnI}
I^\phi (\mu) = 
\left\{
	\begin{array}{ll}
		P(f,\phi) -\int\! \phi \,\mathrm{d}\mu - h_\mu(f)  & \mbox{if } \mu\in\MMM(S^2,f); \\
		+\infty & \mbox{if } \mu\in\PPP(S^2) \setminus \MMM(S^2,f).
	\end{array}
\right.
\end{equation}
Furthermore, for each convex open subset $\GG$ of $\PPP(S^2)$ containing some invariant measure, we have
\begin{equation}  \label{eqInfI}
-\inf\limits_{\GG} I^\phi  = \lim\limits_{n\to+\infty}  \frac{1}{n} \log \Omega_n (x_n)(\GG) =\lim\limits_{n\to+\infty}  \frac{1}{n} \log \Omega_n(\GG)
\end{equation}
and (\ref{eqInfI}) remains true with $\GG$ replaced by its closure $\overline{\GG}$.
\end{theorem}

As an immediate consequence, we get the following corollary. See Section~\ref{subsctEquidistr} for the proof.

\begin{cor}  \label{corMeasTheoPressure}
Let $f\: S^2\rightarrow S^2$ be an expanding Thurston map with no periodic critical points, and $d$ a visual metric on $S^2$ for $f$. Let $\phi$ be a real-valued H\"older continuous function on $(S^2,d)$, and $\mu_\phi$ be the unique equilibrium state for the map $f$ and the potential $\phi$. Given a sequence $\{x_n\}_{n\in\N}$ of points in $S^2$. Fix an arbitrary sequence of functions $\{w_n\: S^2\rightarrow \R\}_{n\in\N}$ satisfying $w_n(x)\in [1,\deg_{f^n}(x)]$ for each $n\in\N$ and each $x\in S^2$.

Then for each $\mu\in\MMM(S^2,f)$, and each convex local basis $\GGG_\mu$ of $\PPP(S^2)$ at $\mu$, we have
\begin{align}   
   & h_\mu(f) + \int\! \phi \,\mathrm{d}\mu  \notag \\
=  & \inf \Bigg\{ \lim\limits_{n\to+\infty} \frac{1}{n} \log \sum\limits_{y\in f^{-n}(x_n), \, W_n(y)\in\GG} w_n(y) e^{S_n\phi(y)} \,\Bigg|\, \GG\in \GGG_\mu \Bigg\}   \label{eqMeasTheoPressure}\\
=  & \inf \Bigg\{ \lim\limits_{n\to+\infty} \frac{1}{n} \log \sum\limits_{x=f^n(x), \, W_n(x)\in\GG} w_n(x) e^{S_n\phi(x)} \,\Bigg|\, \GG\in \GGG_\mu \Bigg\}.   \notag
\end{align}
Here $W_n$ and $S_n\phi$ are as defined in Theorem~\ref{thmLDP}.
\end{cor}

As mentioned above, equidistribution results follow from corresponding level-2 large deviation principles.

\begin{cor}  \label{corEquidistr}
Let $f\: S^2\rightarrow S^2$ be an expanding Thurston map with no periodic critical points, and $d$ a visual metric on $S^2$ for $f$. Let $\phi$ be a real-valued H\"older continuous function on $(S^2,d)$, and $\mu_\phi$ be the unique equilibrium state for the map $f$ and the potential $\phi$. Fix an arbitrary sequence of functions $\{w_n\: S^2\rightarrow \R\}_{n\in\N}$ satisfying $w_n(x)\in [1,\deg_{f^n}(x)]$ for each $n\in\N$ and each $x\in S^2$.

We consider the following sequences of Borel probability measures on $S^2$:

{\bf Iterated preimages:} Given a sequence $\{x_n\}_{n\in\N}$ of points in $S^2$, for each $n\in\N$, put
$$
\nu_n = \sum\limits_{y\in f^{-n}(x_n)} \frac{ w_n(y) \exp(S_n\phi(y))}{\sum_{z\in f^{-n}(x_n)} w_n(z) \exp(S_n\phi(z)) }\frac1n \sum\limits_{i=0}^{n-1} \delta_{f^i(y)},
$$

{\bf Periodic points:} For each $n\in\N$, put
$$
\eta_n  = \sum\limits_{x=f^n(x)} \frac{ w_n(x) \exp(S_n\phi(x))}{\sum_{y=f^n(y)} w_n(y) \exp(S_n\phi(y)) }\frac1n \sum\limits_{i=0}^{n-1} \delta_{f^i(x)}.
$$

Then as $n\longrightarrow +\infty$, 
\begin{equation*}
\nu_n \stackrel{w^*}{\longrightarrow} \mu_\phi,\quad \text{and} \quad
\eta_n \stackrel{w^*}{\longrightarrow} \mu_\phi.
\end{equation*}
Here $S_n$ is defined as in Theorem~\ref{thmLDP}.
\end{cor}

\begin{rem}
Since $S_n\phi(f^i(x))=S_n\phi(x)$ for $i\in\N$ if $f^n(x)=x$, we get
$$
\eta_n  = \sum\limits_{x=f^n(x)} \frac{ \frac{S_n w_n(x)}{n} \exp(S_n\phi(x))}{\sum_{y=f^n(y)} w_n(y) \exp(S_n\phi(y)) }\delta_x,
$$
for $n\in\N$. In particular, when $w_n(\cdot)\equiv 1$, 
$$
\eta_n  = \sum\limits_{x=f^n(x)} \frac{  \exp(S_n\phi(x))}{\sum_{y=f^n(y)} \exp(S_n\phi(y)) }\delta_x;
$$
when $w_n(x)=\deg_{f^n}(x)$, since $\deg_{f^n}(f^i(x))=\deg_{f^n}(x)$ for $i\in\N$ if $f^n(x)=x$, we have
$$
\eta_n  = \sum\limits_{x=f^n(x)} \frac{  \deg_{f^n}(x)\exp(S_n\phi(x))}{\sum_{y=f^n(y)} \deg_{f^n}(y) \exp(S_n\phi(y)) }\delta_x.
$$
\end{rem}

See Section~\ref{subsctEquidistr} for the proof of Corollary~\ref{corEquidistr}. Note that the part of Corollary~\ref{corEquidistr} on iterated preimages generalizes (\ref{eqWeakConvPreImgToMuSumWithWeight}) and (\ref{eqWeakConvPreImgToMuTildeWithWeight}) in Proposition~\ref{propWeakConvPreImgWithWeight} in the context of expanding Thurston maps without periodic critical points. We also remark that our results Corollary~\ref{corUSC} through Corollary~\ref{corEquidistr} are only known in this context. In particular, the following questions for expanding Thurston maps $f\: S^2\rightarrow S^2$ with at least one periodic critical point are still open.

\begin{question}
Is the measure-theoretic entropy $\mu\mapsto h_\mu(f)$ upper semi-continuous?
\end{question}

\begin{question}
Are iterated preimages and periodic points equidistributed with respect to the unique equilibrium state for a H\"older continuous potential?
\end{question}

Note that regarding Question~2, we know that iterated preimages, counted with local degree, are equidistributed with respect to the equilibrium state by (\ref{eqWeakConvPreImgToMuSumWithWeight}) in Proposition~\ref{propWeakConvPreImgWithWeight}. If Question~1 can be answered positively, then the mechanism of Theorem~\ref{thmAbsLargeDeviationPrincipleCRL} works and we get that the equidistribution of periodic points from the corresponding large deviation principle. However, for iterated preimages without counting local degree, (i.e., when $w_n(\cdot)\neq \deg_{f^n}(\cdot)$ in Corollary~\ref{corEquidistr}, and in particular, when $w_n(\cdot)\equiv 1$,) the verification of Condition~(2) mentioned earlier for Theorem~\ref{thmAbsLargeDeviationPrincipleCRL} to apply still remains unknown. Compare (\ref{eqPressureLimitPreImgsWDeg}) and (\ref{eqPressureLimitPreImgsWODeg}) in Proposition~\ref{propPressureLimitPreImgs}.

\smallskip

We will now give a brief description of the structure of this paper.

After fixing some notation in Section~\ref{sctNotation}, we give a quick review of Thurston maps in Section~\ref{sctThurstonMap}. We direct the reader to \cite[Section~3]{Li14} for a more detailed introduction to such maps and the terminology that we use in this paper. However, we do record explicitly most of the results from \cite{BM10,Li13,Li14} that will be used in this paper.

In Section~\ref{sctAssumptions}, we state the assumptions on some of the objects in this paper, which we are going to repeatedly refer to later as \emph{the Assumptions}. Note that these assumptions are the same as those in \cite[Section~4]{Li14}.

Section~\ref{sctAsymHExp} is devoted to the investigation of the weak expansion properties of expanding Thurston maps and the proof of Theorem~\ref{thmWeakExpansion}.

We first introduce basic concepts in Section~\ref{subsctAsymHExpConcepts}. We review the notion of topological conditional entropy $h(g|\lambda)$ of a continuous map $g\: X\rightarrow X$ (on a compact metric space $X$) given an open cover $\lambda$ of $X$, and the notion of topological tail entropy $h^*(g)$ of $g$. The latter was first introduced by M.~Misiurewicz under the name ``topological conditional entropy'' \cite{Mi73,Mi76}. We adopt the terminology and formulations by T.~Downarowicz in \cite{Do11}. We then define $h$-expansiveness and asymptotic $h$-expansiveness using these notions.

In Section~\ref{subsctWeakExpLemmas}, we prove four lemmas that will be used in the proof of the asymptotic $h$-expansiveness of expanding Thurston maps without periodic critical points. Lemma~\ref{lmLocInjectAwayFromCrit} states that any expanding Thurston map is uniformly locally injective away from the critical points, in the sense that if one fixes such a map $f$ and a visual metric $d$ on $S^2$ for $f$, then for each $\delta>0$ sufficiently small and each $x\in S^2$, the map $f$ is injective on the $\delta$-ball centered at $x$ as long as $x$ is not in a $\tau(\delta)$-ball of any critical point of $f$, where $\tau(\delta)$ can be made arbitrarily small  if one lets $\delta$ go to $0$. In Lemma~\ref{lmTileInFlower} we prove a few properties of flowers in the cell decompositions of $S^2$ induced by an expanding Thurston map and some special $f$-invariant Jordan curve. Lemma~\ref{lmCoverByFlowers} gives a covering lemma to cover sets of the form $\bigcap\limits_{i=0}^n f^{-i}(W_i)$ by $(m+n)$-flowers, where $m\in\N_0$, $n\in\N$, and each $W_i$ is an $m$-flower. Finally, we review some basic concepts in graph theory, and provide a simple upper bound of number of leaves of certain trees in Lemma~\ref{lmTree}. Note that we will not use any nontrivial facts from graph theory in this paper.

Section~\ref{subsctProofWeakExp} consists of the proof of Theorem~\ref{thmWeakExpansion} in the form of three separate theorems. Namely, we show in Theorem~\ref{thmAsympHExpWOPC} the asymptotic $h$-expansiveness of expanding Thurston maps without periodic critical points. The proof relies on a quantitative upper bound of the frequency for an orbit under such a map to get close to the set of critical points. Lemma~\ref{lmTree} and terminology from graph theory is used here to make the statements in the proof precise. We then prove in Theorem~\ref{thmNotAsympHExp} and Theorem~\ref{thmNotHExp} the lack of asymptotic $h$-expansiveness of expanding Thurston maps with periodic critical points and the lack of $h$-expansiveness of expanding Thurston maps without periodic critical points, respectively, by explicit constructions of periodic sequences $\{v_i\}_{i\in\N}$ of $m$-vertices for which one can give lower bounds for the numbers of open sets in the open cover $\bigvee\limits_{j=0}^{n-1} f^{-j}\(\W^m\)$ needed to cover the set $\bigcap\limits_{j=0}^{n-1} f^{-j}(W^m(v_{n-j}))$, for $l,m,n\in\N$ sufficiently large. Here $W^m(v_{n-j})$ denotes the $m$-flower of $v_{n-j}$ (see (\ref{defFlower})), and $\W^m$ is the set of all $m$-flowers (see (\ref{defSetNFlower})). These lower bounds lead to the conclusion that the topological tail entropy and topological conditional entropy, respectively, are strictly positive, proving the corresponding theorems (compare with Defintion~\ref{defHExp} and Definition~\ref{defAsympHExp}). The periodic sequence $\{v_i\}_{i\in\N}$ of $m$-vertices in the proof of Theorem~\ref{thmNotAsympHExp} shadows a certain infinite backward pseudo-orbit in such a way that each period of $\{v_i\}_{i\in\N}$ begins with a backward orbit starting at a critical point $p$ which is a fixed point of $f$, and approaching $p$ as the index $i$ increases, and then ends with a constant sequence staying at $p$. The fact that the constant part of each period of $\{v_i\}_{i\in\N}$ can be made arbitrarily long is essential here and is not true if $f$ has no periodic critical points. The periodic sequence $\{v_i\}_{i\in\N_0}$ of $m$-vertices in the proof of Theorem~\ref{thmNotHExp} shadows a certain infinite backward pseudo-orbit in such a way that each period of $\{v_i\}_{i\in\N_0}$ begins with a backward orbit starting at $f(p)$ and $p$, and approaching $f(p)$ as the index $i$ increases, and then ends with $f(p)$. In this case $p$ is a critical point whose image $f(p)$ is a fixed point. In both constructions, we may need to consider an iterate of $f$ for the existence of $p$ with the required properties. Combining Theorems~\ref{thmAsympHExpWOPC}, \ref{thmNotAsympHExp}, and \ref{thmNotHExp}, we get Theorem~\ref{thmWeakExpansion}.

Section~\ref{sctLDP} is devoted to the study of large deviation principles and equidistribution results for periodic points and iterated preimages of expanding Thurston maps without periodic critical points. The idea is to apply a general framework devised by Y.~Kifer \cite{Ki90} to obtain level-2 large deviation principles, and to derive the equidistribution results as consequences.

In Section~\ref{subsctThermDynForm}, we review briefly the theory of thermodynamical formalism and recall relevant concepts and results in this theory from \cite{Li14} in the context of expanding Thurston maps and H\"older continuous potentials. After the necessary concepts are introduced, we provide a quick proof of Theorem~\ref{thmExistenceES}, which asserts the existence of equilibrium states for expanding Thurston maps without periodic critical points and given continuous potentials.

In Section~\ref{subsctLDP}, we give a brief review of level-2 large deviation principles in our context. We record the theorem of Y.~Kifer \cite{Ki90}, reformulated by H.~Comman and J.~Rivera-Letelier \cite{CRL11}, on level-2 large deviation principles. This result, stated in Theorem~\ref{thmAbsLargeDeviationPrincipleCRL}, will be applied later to our context.

After proving and recording several technical lemmas in Section~\ref{subsctLDPLemmas}, we generalize some characterization of topological pressure in Section~\ref{subsctTopPres} in our context. More precisely, we use equidistribution results for iterated preimages from \cite{Li14} recorded in Proposition~\ref{propWeakConvPreImgWithWeight} to show in Proposition~\ref{propPressureLimitPreImgs} and Proposition~\ref{propPressureLimitPeriodicPts} that
\begin{equation}   \label{eqTopPresCharact}
P(f,\phi)=\lim\limits_{n\to +\infty} \frac{1}{n} \log \sum w_n(y) \exp(S_n\phi(y)),
\end{equation}
where the sum is taken over preimages under $f^n$ in Proposition~\ref{propPressureLimitPreImgs}, and over periodic points in Proposition~\ref{propPressureLimitPeriodicPts}, the potential $\phi\:S^2 \rightarrow \R$ is H\"older continuous with respect to a visual metric $d$, and the weight $w_n(y)\in [1,\deg_{f^n}(y)]$ for $n\in\N$ and $y\in S^2$. We note that for periodic points, the equation (\ref{eqTopPresCharact}) is established in Proposition~\ref{propPressureLimitPeriodicPts} for all expanding Thurston maps, but for iterated preimages,  we only obtain (\ref{eqTopPresCharact}) for expanding Thurston maps without periodic critical points in Proposition~\ref{propPressureLimitPreImgs}.

In Section~\ref{subsctProofLDP}, by applying Theorem~\ref{thmAbsLargeDeviationPrincipleCRL} to give a proof of Theorem~\ref{thmLDP}, we finally establish level-2 large deviation principles in the context of expanding Thurston maps without periodic critical points and given H\"older continuous potentials. 

Section~\ref{subsctEquidistr} consists of the proofs of Corollary~\ref{corMeasTheoPressure} and Corollary~\ref{corEquidistr}. We first obtain characterizations of the measure-theoretic pressure in terms of the infimum of certain limits involving periodic points and iterated preimages (Corollary~\ref{corMeasTheoPressure}). Such characterizations are then used in the proof of the equidistribution results (Corollary~\ref{corEquidistr}).

\bigskip
\noindent
\textbf{Acknowledgments.} The author wants to express his gratitude to the Institut Henri Poincar\'e for the kind hospitality during his stay in Paris from January to March 2014, when a major part of this work was carried out. The author also would like to thank N.-P.~Chung for explaining his work on weak expansiveness for actions of sofic groups. Last but not least, the author wants to express his deepest gratitude to M.~Bonk for his patient teaching and guidance as the advisor of the author.

\section{Notation} \label{sctNotation}
Let $\C$ be the complex plane and $\widehat{\C}$ be the Riemann sphere. We use the convention that $\N=\{1,2,3,\dots\}$ and $\N_0 = \{0\} \cup \N$. As usual, the symbol $\log$ denotes the logarithm to the base $e$.

The cardinality of a set $A$ is denoted by $\card{A}$. For $x\in\R$, we define $\lfloor x\rfloor$ as the greatest integer $\leq x$, and $\lceil x \rceil$ the smallest integer $\geq x$.

Let $g\: X\rightarrow Y$ be a function between two sets $X$ and $Y$. We denote the restriction of $g$ to a subset $Z$ of $X$ by $g|_Z$. 

Let $(X,d)$ be a metric space. For subsets $A,B\subseteq X$, we set $d(A,B)=\inf \{d(x,y)\,|\, x\in A,\,y\in B\}$, and $d(A,x)=d(x,A)=d(A,\{x\})$ for $x\in X$. For each subset $Y\subseteq X$, we denote the diameter of $Y$ by $\diam_d(Y)=\sup\{d(x,y)\,|\,x,y\in Y\}$, the interior of $Y$ by $\inter Y$, the closure of $Y$ by $\overline Y$, and the characteristic function of $Y$ by $\mathbbm{1}_Y$, which maps each $x\in Y$ to $1\in\R$. For each $r>0$, we define $N^r_d(A)$ to be the open $r$-neighborhood $\{y\in X \,|\, d(y,A)<r\}$ of $A$, and $\overline{N^r_d}(A)$ the closed $r$-neighborhood $\{y\in X \,|\, d(y,A)\leq r\}$ of $A$. For $x\in X$, we denote the open ball of radius $r$ centered at $x$ by $B_d(x, r)$. 

We set $\CCC(X)$ to be the space of continuous functions from $X$ to $\R$, by $\MMM(X)$ the set of finite signed Borel measures, and $\PPP(X)$ the set of Borel probability measures on $X$. For $\mu\in\MMM(X)$, we use $\Norm{\mu}$ to denote the total variation norm of $\mu$, $\supp \mu$ the support of $\mu$, and
$$
\langle \mu,u \rangle = \int \! u \,\mathrm{d}\mu
$$
for each $u\in\CCC(S^2)$. For a point $x\in X$, we define $\delta_x$ as the Dirac measure supported on $\{x\}$. For $g\in\CCC(X)$ we set $\MMM(X,g)$ to be the set of $g$-invariant Borel probability measures on $X$. If we do not specify otherwise, we equip $\CCC(X)$ with the uniform norm $\Norm{\cdot}_\infty$, and equip both $\MMM(X)$ and $\MMM(X,g)$ with the weak$^*$ topology.

The space of real-valued H\"{o}lder continuous functions with an exponent $\alpha\in (0,1]$ on a compact metric space $(X,d)$ is denoted as $\Holder{\alpha}(X,d)$. For given $f\: X \rightarrow X$ and $\varphi \in \CCC(X)$, we define
\begin{equation}    \label{eqDefSnPt}
S_n \varphi (x)  = \sum\limits_{j=0}^{n-1} \varphi(f^j(x)) 
\end{equation}
and
\begin{equation}  \label{eqDefWn}
W_n(x) = \frac{1}{n} \sum\limits_{j=0}^{n-1} \delta_{f^j(x)}
\end{equation}
for $x\in X$ and $n\in\N_0$. Note that when $n=0$, by definition we always have $S_0 \varphi = 0$, and by convention $W_0=0$.

\section{Thurston maps} \label{sctThurstonMap}

This section serves as a minimal review for expanding Thurston maps. Most of the definitions and results here were discussed in \cite[Section~3]{Li14}. The reader is encouraged to read Section~3 in \cite{Li14} for a quick introduction to expanding Thurston maps and the terminology that we use in this paper. For a more thorough treatment of the subject, we refer to \cite{BM10}.

Let $S^2$ denote an oriented topological $2$-sphere. A continuous map $f\:S^2\rightarrow S^2$ is called a \defn{branched covering map} on $S^2$ if for each point $x\in S^2$, there exists a positive integer $d\in \N$, open neighborhoods $U$ of $x$ and $V$ of $y=f(x)$, open neighborhoods $U'$ and $V'$ of $0$ in $\widehat{\C}$, and orientation-preserving homeomorphisms $\varphi\:U\rightarrow U'$ and $\eta\:V\rightarrow V'$ such that $\varphi(x)=0$, $\eta(y)=0$, and
$$
(\eta\circ f\circ\varphi^{-1})(z)=z^d
$$
for each $z\in U'$. The positive integer $d$ above is called the \defn{local degree} of $f$ at $x$ and is denoted by $\deg_f (x)$. The \defn{degree} of $f$ is
\begin{equation}   \label{eqDeg=SumLocalDegree}
\deg f=\sum\limits_{x\in f^{-1}(y)} \deg_f (x)
\end{equation}
for $y\in S^2$ and is independent of $y$. If $f\:S^2\rightarrow S^2$ and $g\:S^2\rightarrow S^2$ are two branched covering maps on $S^2$, then so is $f\circ g$, and
\begin{equation} \label{eqLocalDegreeProduct}
 \deg_{f\circ g}(x) = \deg_g(x)\deg_f(g(x)), \qquad \text{for each } x\in S^2.
\end{equation}   

A point $x\in S^2$ is a \defn{critical point} of $f$ if $\deg_f(x) \geq 2$. The set of critical points of $f$ is denoted by $\crit f$. A point $y\in S^2$ is a \defn{postcritical point} of $f$ if $y = f^n(x)$ for some $x\in\crit f$ and $n\in\N$. The set of postcritical points of $f$ is denoted by $\post f$. Note that $\post f=\post f^n$ for all $n\in\N$.

\begin{definition} [Thurston maps] \label{defThurstonMap}
A Thurston map is a branched covering map $f\:S^2\rightarrow S^2$ on $S^2$ with $\deg f\geq 2$ and $\card(\post f)<+\infty$.
\end{definition}

Let $f\:S^2 \rightarrow S^2$ be a Thurston map, and $\CC\subseteq S^2$ be a Jordan curve containing $\post f$. Then the pair $f$ and $\CC$ induces natural \emph{cell decompositions} (see \cite[Definition~3.2]{Li14}) $\DD^n(f,\CC)$ of $S^2$, for $n\in\N_0$, such that
$$
\DD^n(f,\CC)=\X^n(f,\CC) \cup \E^n(f,\CC) \cup \overline\V^n(f,\CC)
$$
consisting of \emph{$n$-cells}, where the set $\X^n(f,\CC)$ consists of \emph{$n$-tiles}, the set $\E^n(f,\CC)$ consists of \emph{$n$-edges}, and $\overline\V^n(f,\CC) = \{ \{x\} \,|\, x\in \V^n(f,\CC)\}$ where the set $\V^n(f,\CC)$ consists of \emph{$n$-vertices}. The \defn{interior} of an $n$-cell is denoted by $\inte(c)$ (see the discussion preceding Definition~3.2 in \cite{Li14}). The \defn{$k$-skeleton}, for $k\in\{0,1,2\}$, of $\DD^n(f,\CC)$ is the union of all $n$-cells of dimension $k$ in this cell decomposition.

We record Proposition~6.1 of \cite{BM10} here in order to summarize properties of the cell decompositions $\DD^n(f,\CC)$ defined above.

\begin{prop}[M.~Bonk \& D.~Meyer, 2010] \label{propCellDecomp}
Let $k,n\in \N_0$, let   $f\: S^2\rightarrow S^2$ be a Thurston map,  $\CC\subseteq S^2$ be a Jordan curve with $\post f \subseteq \CC$, and   $m=\card(\post f)$. 
 
\smallskip
\begin{itemize}

\smallskip
\item[(i)] The map  $f^k$ is cellular for $(\DD^{n+k}(f,\CC), \DD^n(f,\CC))$. In particular, if  $c$ is any $(n+k)$-cell, then $f^k(c)$ is an $n$-cell, and $f^k|_c$ is a homeomorphism of $c$ onto $f^k(c)$.

\smallskip
\item[(ii)]  Let  $c$ be  an $n$-cell.  Then $f^{-k}(c)$ is equal to the union of all 
$(n+k)$-cells $c'$ with $f^k(c')=c$.

\smallskip
\item[(iii)] The $1$-skeleton of $\DD^n(f,\CC)$ is  equal to  $f^{-n}(\CC)$. The $0$-skeleton of $\DD^n(f,\CC)$ is the set $\V^n(f,\CC)=f^{-n}(\post f )$, and we have $\V^n(f,\CC) \subseteq \V^{n+k}(f,\CC)$. 

\smallskip
\item[(iv)] $\card(\X^n(f,\CC))=2(\deg f)^n$,  $\card(\E^n(f,\CC))=m(\deg f)^n$,  and $\card (\V^n(f,\CC)) \leq m (\deg f)^n$.

\smallskip
\item[(v)] The $n$-edges are precisely the closures of the connected components of $f^{-n}(\CC)\setminus f^{-n}(\post f )$. The $n$-tiles are precisely the closures of the connected components of $S^2\setminus f^{-n}(\CC)$.

\smallskip
\item[(vi)] Every $n$-tile  is an $m$-gon, i.e., the number of $n$-edges and the number of $n$-vertices contained in its boundary are equal to $m$.  

\end{itemize}
\end{prop}

From now on, if the map $f$ and the Jordan curve $\CC$ are clear from the context, we will sometimes omit $(f,\CC)$ in the notation above.

If we fix the cell decomposition $\DD^n(f,\CC)$, $n\in\N_0$, we can define for each $v\in \V^n$ the \defn{$n$-flower of $v$} as
\begin{equation}   \label{defFlower}
W^n(v) = \bigcup  \{\inte (c) \,|\, c\in \DD^n,\, v\in c \}.
\end{equation}
Note that flowers are open (in the standard topology on $S^2$). Let $\overline{W}^n(v)$ be the closure of $W^n(v)$. We define the \defn{set of all $n$-flowers} by
\begin{equation}   \label{defSetNFlower}
\W^n = \{W^n(v) \,|\, v\in\V^n\}.
\end{equation}
\begin{rem}  \label{rmFlower}
For $n\in\N_0$ and $v\in\V^n$, we have 
$$
\overline{W}^n(v)=X_1\cup X_2\cup \cdots \cup X_m,
$$
where $m=2\deg_{f^n}(v)$, and $X_1, X_2, \dots X_m$ are all the $n$-tiles that contains $v$ as a vertex (see \cite[Lemma~7.2]{BM10}). Moreover, each flower is mapped under $f$ to another flower in such a way that is similar to the map $z\mapsto z^k$ on the complex plane. More precisely, for $n\in\N_0$ and $v\in \V^{n+1}$, there exists orientation preserving homeomorphisms $\varphi\: W^{n+1}(v) \rightarrow D$ and $\eta\: W^{n}(f(v)) \rightarrow D$ such that $D$ is the unit disk on $\C$, $\varphi(v)=0$, $\eta(f(v))=0$, and 
$$
(\eta\circ f \circ \varphi^{-1}) (z) = z^k
$$
for all $z\in D$, where $k=\deg_f(v)$. Let $\overline{W}^{n+1}(v)= X_1\cup X_2\cup \cdots \cup X_m$ and $\overline{W}^n(f(v))= X'_1\cup X'_2\cup \cdots \cup X'_{m'}$, where $X_1, X_2, \dots X_m$ are all the $(n+1)$-tiles that contains $v$ as a vertex, listed counterclockwise, and $X'_1, X'_2, \dots X'_{m'}$ are all the $n$-tiles that contains $f(v)$ as a vertex, listed counterclockwise, and $f(X_1)=X'_1$. Then $m= m'k$, and $f(X_i)=X'_j$ if $i\equiv j \pmod{k}$, where $k=\deg_f(v)$. (See also Case~3 of the proof of Lemma~5.2 in \cite{BM10} for more details.)
\end{rem}

\begin{definition} [Expansion] \label{defExpanding}
A Thurston map $f\:S^2\rightarrow S^2$ is called \defn{expanding} if there exist a metric $d$ on $S^2$ that induces the standard topology on $S^2$ and a Jordan curve $\CC\subseteq S^2$ containing $\post f$ such that 
\begin{equation*}
\lim\limits_{n\to+\infty}\max \{\diam_d(X) \,|\, X\in \X^n(f,\CC)\}=0.
\end{equation*}
\end{definition}

It is clear that if $f\: S^2\rightarrow S^2$ is an expanding Thurston map, then so is $f^n\: S^2\rightarrow S^2$, for $n\in \N$.

For an expanding Thurston map $f$, we can fix a particular metric $d$ on $S^2$ called a \emph{visual metric for $f$} . For the existence and properties of such metrics, see \cite[Chapter~8]{BM10}. For a fixed expanding Thurston map, each visual metric corresponds to a unique \emph{expansion factor} $\Lambda >1$. One major advantage of a visual metric $d$ is that in $(S^2,d)$ we have good quantitative control over the sizes of the cells in the cell decompositions discussed above (see \cite[Lemma~8.10]{BM10}).

\begin{lemma}[M.~Bonk \& D.~Meyer, 2010]   \label{lmCellBoundsBM}
Let $f\:S^2 \rightarrow S^2$ be an expanding Thurston map, and $\CC \subseteq S^2$ be a Jordan curve containing $\post f$. Let $d$ be a visual metric on $S^2$ for $f$ with expansion factor $\Lambda>1$. Then there exists a constant $C\geq 1$ such that for all $n$-edges and all $n$-tiles $\tau$ with $n\in\N_0$, we have $C^{-1} \Lambda^{-n} \leq \diam_d(\tau) \leq C\Lambda^{-n}$.
\end{lemma}

In addition, we will need the fact that a visual metric $d$ induces the standard topology on $S^2$ (\cite[Proposition~8.9]{BM10}) and the fact that the metric space $(S^2,d)$ is \emph{linearly locally connected} (\cite[Proposition~16.3]{BM10}).

A Jordan curve $\CC\subseteq S^2$ is \defn{$f$-invariant} if $f(\CC)\subseteq \CC$. For each $f$-invariant Jordan curve $\CC\subseteq S^2$ containing $\post f$, the partition $(\DD^1,\DD^0)$ is a \emph{cellular Markov partition} for $f$ (see \cite[Definition~3.4]{Li14}). M.~Bonk and D.~Meyer \cite[Theorem~1.2]{BM10} proved that there exists an $f^n$-invariant Jordan curve $\CC\subseteq S^2$ containing $\post{f}$ for each sufficiently large $n$ depending on $f$. We proved a slightly stronger version of this result in \cite[Lemma~3.12]{Li13} which we record in the following lemma.

\begin{lemma}  \label{lmCexistsL}
Let $f\:S^2\rightarrow S^2$ be an expanding Thurston map, and $\widetilde{\CC}\subseteq S^2$ be a Jordan curve with $\post f\subseteq \widetilde{\CC}$. Then there exists an integer $N(f,\widetilde{\CC}) \in \N$ such that for each $n\geq N(f,\widetilde{\CC})$ there exists an $f^n$-invariant Jordan curve $\CC$ isotopic to $\widetilde{\CC}$ rel.\ $\post f$ such that no $n$-tile in $\DD^n(f,\CC)$ joins opposite sides of $\CC$.
\end{lemma}

\begin{definition}[Joining opposite sides]  \label{defConnectop} 
Fix a Thurston map $f$ with $\card(\post f) \geq 3$ and an $f$-invariant Jordan curve $\CC$ containing $\post f$.  A set $K\subseteq S^2$ \defn{joins opposite sides} of $\CC$ if $K$ meets two disjoint $0$-edges when $\card( \post f)\geq 4$, or $K$ meets  all  three $0$-edges when $\card(\post f)=3$. 
 \end{definition}
 
Note that $\card (\post f) \geq 3$ for each expanding Thurston map $f$ \cite[Corollary~6.4]{BM10}.

We proved in \cite[Lemma~3.14]{Li13} the following easy lemma. 

\begin{lemma}   \label{lmPreImageDense}
Let $f\:S^2\rightarrow S^2$ be an expanding Thurston map. Then for each $p\in S^2$, the set $\bigcup\limits_{n=1}^{+\infty}  f^{-n}(p)$ is dense in $S^2$, and
\begin{equation}
\lim\limits_{n\to +\infty}  \card(f^{-n}(p))  = +\infty.
\end{equation}
\end{lemma}

Expanding Thurston maps are Lipschitz with respect to a visual metric \cite[Lemma~3.12]{Li14}.

\begin{lemma}    \label{lmLipschitz}
Let $f\:S^2 \rightarrow S^2$ be an expanding Thurston map, and $d$ be a visual metric on $S^2$ for $f$. Then $f$ is Lipschitz with respect to $d$.
\end{lemma}

We established the following generalization of \cite[Lemma~16.1]{BM10} in \cite[Lemma~3.13]{Li14}.

\begin{lemma}  \label{lmMetricDistortion}
Let $f\:S^2 \rightarrow S^2$ be an expanding Thurston map, and $\CC \subseteq S^2$ be a Jordan curve that satisfies $\post f \subseteq \CC$ and $f^{n_\CC}(\CC)\subseteq\CC$ for some $n_\CC\in\N$. Let $d$ be a visual metric on $S^2$ for $f$ with expansion factor $\Lambda>1$. Then there exists a constant $C_0 > 1$, depending only on $f$, $d$, $\CC$, and $n_\CC$, with the following property:

If $k,n\in\N_0$, $X^{n+k}\in\X^{n+k}(f,\CC)$, and $x,y\in X^{n+k}$, then 
\begin{equation}   \label{eqMetricDistortion}
\frac{1}{C_0} d(x,y) \leq \frac{d(f^n(x),f^n(y))}{\Lambda^n}  \leq C_0 d(x,y).
\end{equation}
\end{lemma}

\section{The Assumptions}      \label{sctAssumptions}
We state below the hypothesis under which we will develop our theory in most parts of this paper. We will repeatedly refer to such assumptions in the later sections.

\begin{assumptions}
\quad

\begin{enumerate}

\smallskip

\item $f\:S^2 \rightarrow S^2$ is an expanding Thurston map.

\smallskip

\item $\CC\subseteq S^2$ is a Jordan curve containing $\post f$ with the property that there exists $n_\CC\in\N$ such that $f^{n_\CC} (\CC)\subseteq \CC$ and $f^m(\CC)\nsubseteq \CC$ for each $m\in\{1,2,\dots,n_\CC-1\}$.

\smallskip

\item $d$ is a visual metric on $S^2$ for $f$ with expansion factor $\Lambda>1$ and a linear local connectivity constant $L\geq 1$.

\smallskip

\item $\phi\in \Holder{\alpha}(S^2,d)$ is a real-valued H\"{o}lder continuous function with an exponent $\alpha\in(0,1]$.

\end{enumerate}

\end{assumptions}

Observe that by Lemma~\ref{lmCexistsL}, for each $f$ in (1), there exists at least one Jordan curve $\CC$ that satisfies (2). Since for a fixed $f$, the number $n_\CC$ is uniquely determined by $\CC$ in (2), in the remaining part of the paper we will say that a quantity depends on $\CC$ even if it also depends on $n_\CC$.

Recall that the expansion factor $\Lambda$ of a visual metric $d$ on $S^2$ for $f$ is uniquely determined by $d$ and $f$. We will say that a quantity depends on $f$ and $d$ if it depends on $\Lambda$.

Note that even though the value of $L$ is not uniquely determined by the metric $d$, in the remainder of this paper, for each visual metric $d$ on $S^2$ for $f$, we will fix a choice of linear local connectivity constant $L$. We will say that a quantity depends on the visual metric $d$ without mentioning the dependence on $L$, even though if we had not fixed a choice of $L$, it would have depended on $L$ as well.

In the discussion below, depending on the conditions we will need, we will sometimes say ``Let $f$, $\CC$, $d$, $\phi$, $\alpha$ satisfy the Assumptions.'', and sometimes say ``Let $f$ and $d$ satisfy the Assumptions.'', etc.

\section{Asymptotic $h$-Expansiveness}  \label{sctAsymHExp}

\subsection{Basic concepts}  \label{subsctAsymHExpConcepts}
We first review some concepts from dynamical systems. We refer the reader to \cite[Chapter~3]{PU10}, \cite[Chapter~9]{Wa82} or \cite[Chapter~20]{KH95} for more detailed studies of these concepts.

Let $(X,d)$ be a compact metric space and $g\:X\rightarrow X$ a continuous map. 

A \defn{cover} of $X$ is a collection $\xi=\{A_j \,|\, j\in J\}$ of subsets of $X$ with the property that $\bigcup\xi = X$, where $J$ is an index set. The cover $\xi$ is an \defn{open cover} if $A_j$ is an open set for each $j\in J$. The cover $\xi$ is \defn{finite} if the index set $J$ is a finite set.

A \defn{measurable partition} $\xi$ of $X$ is a cover $\xi=\{A_j\,|\,j\in J\}$ of $X$ consisting of countably many mutually disjoint Borel sets $A_j$, $j\in J$, where $J$ is a countable index set. 

Let $\xi=\{A_j\,|\,j\in J\}$ and $\eta=\{B_k\,|\,k\in K\}$ be two covers of $X$, where $J$ and $K$ are the corresponding index sets. We say $\xi$ is a \defn{refinement} of $\eta$ if for each $A_j\in\xi$, there exists $B_k\in\eta$ such that $A_j\subseteq B_k$. The \defn{common refinement} $\xi \vee \eta$ of $\xi$ and $\eta$ defined as
$$
\xi \vee \eta = \{A_j\cap B_k \,|\, j\in J,\, k\in K\}
$$
is also a cover. Note that if $\xi$ and $\eta$ are both open covers (resp., measurable partitions), then $\xi \vee \eta$ is also an open cover (resp., a measurable partition). Define $g^{-1}(\xi)=\{g^{-1}(A_j) \,|\,j\in J\}$, and denote for $n\in\N$,
$$
\xi^n_g= \bigvee\limits_{j=0}^{n-1} g^{-j}(\xi) = \xi\vee g^{-1}(\xi)\vee\cdots\vee g^{-(n-1)}(\xi).
$$

We adopt the following definition from \cite[Remark~6.1.7]{Do11}.

\begin{definition}  [Refining sequences of open covers]    \label{defRefSeqOpenCover}
A sequence of open covers $\{\xi_i\}_{i\in\N_0}$ of a compact metric space $X$ is a \defn{refining sequence of open covers} of $X$ if the following conditions are satisfied
\begin{enumerate}
\smallskip
\item[(i)] $\xi_{i+1}$ is a refinement of $\xi_i$ for each $i\in\N_0$.

\smallskip
\item[(ii)] For each open cover $\eta$ of $X$, there exists $j\in\N$ such that $\xi_i$ is a refinement of $\eta$ for each $i\geq j$.
\end{enumerate} 
\end{definition}

By the Lebesgue Number Lemma (\cite[Lemma~27.5]{Mu00}), it is clear that for a compact metric space, refining sequences of open covers always exist.

The topological tail entropy was first introduced by M.~Misiurewicz under the name ``topological conditional entropy'' \cite{Mi73, Mi76}. We adopt the terminology in \cite{Do11} (see \cite[Remark~6.3.18]{Do11}).

\begin{definition}[Topological conditional entropy and topological tail entropy]   \label{defTopTailEntropy}
Let $(X,d)$ be a compact metric space and $g\:X\rightarrow X$ a continuous map. The \defn{topological conditional entropy} $h(g|\lambda)$ of $g$ given $\lambda$, for some open cover $\lambda$, is
\begin{equation}  \label{eqDefTopCondEntropy}
h(g|\lambda)= \lim\limits_{l\to+\infty} \lim\limits_{n\to+\infty}  \frac{1}{n} H \( \bigvee\limits_{i=0}^{n-1} g^{-i}\(\xi_l\) \Bigg| \bigvee\limits_{j=0}^{n-1} g^{-j}\(\lambda\)  \),
\end{equation}
where $\{\xi_l\}_{l\in\N_0}$ is an arbitrary refining sequence of open covers, and for each pair of open covers $\xi$ and $\eta$,
\begin{equation} \label{eqH(xi|eta)}
H(\xi | \eta) = \log \Big(\max_{A\in\eta} \Big\{\min\Big\{\card \xi_A \,\Big|\, \xi_A  \subseteq \xi,\, A \subseteq \bigcup \xi_A \Big\}\Big\} \Big)
\end{equation}
is the logarithm of the minimal number of sets from $\xi$ sufficient to cover any set in $\eta$. 

The \defn{topological tail entropy} $h^*(g)$ of $g$ is defined by
\begin{equation}   \label{eqDefTopTailEntropy}
h^*(g) = \lim\limits_{m\to+\infty}  \lim\limits_{l\to+\infty} \lim\limits_{n\to+\infty}  \frac{1}{n} H \( \bigvee\limits_{i=0}^{n-1} g^{-i}\(\xi_l\) \Bigg| \bigvee\limits_{j=0}^{n-1} g^{-j}\(\eta_m\)  \),
\end{equation} 
where $\{\xi_l\}_{l\in\N_0}$ and $\{\eta_m\}_{m\in\N_0}$ are two arbitrary refining sequences of open covers, and $H$ is as defined in (\ref{eqH(xi|eta)}).
\end{definition}

\begin{remark}
The topological entropy of $g$ (see Section~\ref{subsctThermDynForm}) is $h_{\operatorname{top}}(g) = h(g|\{X\})$, where $\{X\}$ is the open cover of $X$ consisting of only one open set $X$. See for example, \cite[Section~6.1]{Do11}.
\end{remark}

The limits in (\ref{eqDefTopCondEntropy}) and (\ref{eqDefTopTailEntropy}) always exist, and both $h(g|\lambda)$ and $h^*(g)$ are independent of the choices of refining sequences of open covers $\{\xi_l\}_{l\in\N_0}$ and $\{\eta_m\}_{m\in\N_0}$, see \cite[Section~6.3]{Do11}, especially the comments after \cite[Definition~6.3.14]{Do11}.

The topological tail entropy $h^*$ is also well-behaved under iterations, as it satisfies 
\begin{equation}   \label{eqH*gn=nH*g}
h^*(g^n) = n h^*(g)
\end{equation}
for each $n\in\N$ and each continuous map $g\:X\rightarrow X$ on a compact metric space $X$ (\cite[Proposition~3.1]{Mi76}).

The concept of $h$-expansiveness was introduced by R.~Bowen in \cite{Bow72}. We adopt the formulation in \cite{Mi76} (see also \cite{Do11}). 

\begin{definition}[$h$-expansiveness]    \label{defHExp}
A continuous map $g\: X\rightarrow X$ on a compact metric space $X$ is called \defn{$h$-expansive} if there exists a finite open cover $\lambda$ of $X$ such that $h(g|\lambda)=0$.
\end{definition}

A weaker property was then introduced by M.~Misiurewicz in \cite{Mi73} (see also \cite{Mi76, Do11}).

\begin{definition}[Asymptotic $h$-expansiveness]    \label{defAsympHExp}
We say that a continuous map $g\: X\rightarrow X$ on a compact metric space $X$ is \defn{asymptotically $h$-expansive} if $h^*(g)=0$.
\end{definition}

\subsection{Technical lemmas}  \label{subsctWeakExpLemmas}

Now we go back to the dynamical system $(S^2,f)$ where $f$ is an expanding Thurston map.

We need the following four lemmas for the proof of the asymptotic $h$-expansiveness of expanding Thurston maps with no periodic critical points.

\begin{lemma}[Uniform local injectivity away from the critical points]   \label{lmLocInjectAwayFromCrit}
Let $f$, $d$ satisfies the Assumptions. Then there exists a number $\delta_0 \in (0,1]$ and a function $\tau\: (0,\delta_0]\rightarrow (0,+\infty)$ with the following properties:
\begin{enumerate}
\smallskip
\item[(i)] $\lim\limits_{\delta \to 0} \tau(\delta) = 0$.

\smallskip
\item[(ii)] For each $\delta \leq \delta_0$, the map $f$ restricted to any open ball of radius $\delta$ centered outside the $\tau(\delta)$-neighborhood of $\crit f$ is injective, i.e., $f|_{B_d(x,\delta)}$ is injective for each $x \in S^2 \setminus N_d^{\tau(\delta)}(\crit f)$.
\end{enumerate}
\end{lemma}

This lemma is straightforward to verify, but for the sake of completeness, we include the proof here.

\begin{proof}
We first define a function $r\: S^2\setminus \crit f \rightarrow (0,+\infty)$ in the following way
\begin{equation*}
r(x) = \sup \{ R>0 \,|\, f|_{B_d(x,R)} \text{ is injective} \},
\end{equation*}
for $x\in S^2\setminus \crit f$. Note that $r(x) \leq d(x,\crit f) <+\infty$ for each $x\in S^2\setminus \crit f$. We also observe that the supremum is attained, since otherwise, suppose $f(y)=f(z)$ for some $y,z\in B(x,r(x))$, then $f$ is not injective on the ball $B(x,R_0)$ containing $y$ and $z$ with $R_0 = \frac12(r(x) + \max\{d(x,y),\,d(x,z)\}) <r(x)$, a contradiction.

We claim that $r$ is continuous.

Indeed, let $\{x_i\}_{i\in\N}$ be a sequence of points in $S^2$ and $x\in S^2$ with the property that $\lim\limits_{i\to+\infty} x_i=x$. For each $i\in\N$, if $r(x_i) - d(x_i,x)>0$, then $B(x,r(x_i)-d(x_i,x)) \subseteq B(x_i,r(x_i))$. So $r(x) \geq r(x_i) - d(x_i,x)$. Thus
$$
r(x) \geq \limsup\limits_{i\to+\infty} (r(x_i)-d(x_i,x)) =\limsup\limits_{i\to+\infty} r(x_i).
$$
On the other hand, for each $i\in \N$, if $r(x_i) - d(x_i,x)>0$, then $B(x_i,r(x)-d(x_i,x)) \subseteq B(x,r(x))$. So $r(x_i) \geq r(x) - d(x_i,x)$. Thus
$$
\liminf\limits_{i\to+\infty} r(x_i) \geq \liminf\limits_{i\to+\infty}(r(x)-d(x_i,x)) =r(x).
$$
Hence $r(x)=\lim\limits_{i\to+\infty} r(x_i)$. So $r$ is continuous and the claim is proved.

Next, we fix a sufficiently small number $t_0>0$ with $S^2\setminus N_d^{t_0} (\crit f) \neq \emptyset$. We define a function $\sigma \: (0,t_0] \rightarrow (0,+\infty)$ by setting
\begin{equation*}
\sigma(t) = \inf \{ r(x) \,|\, x\in S^2\setminus N_d^t (\crit f)  \}
\end{equation*}
for $t\in (0,t_0]$. We observe that $\sigma$ is continuous and non-decreasing. Since $r(x)\leq d(x,\crit f)$ for each $x\in S^2\setminus \crit f$, we can conclude that $\lim\limits_{t\to 0} \sigma(t) = 0$. By the definition of $\sigma$, we get that $f|_{B_d(x,\sigma(t))}$ is injective, for $t\in (0,t_0]$ and $x\in S^2\setminus N_d^t(\crit f)$.

Finally, we construct $\tau\: (0,\delta_0]\rightarrow (0,+\infty)$, where $\delta_0 = \min \{1,\sigma(t_0)\}$ by setting
\begin{equation}   \label{eqTau}
\tau (\delta) = \inf \{ t\in (0,t_0] \,|\, \sigma(t) \geq \delta \}
\end{equation}
for each $\delta\in(0,\delta_0]$. We note that $\lim\limits_{\delta \to 0}  \tau(\delta) = 0$.

For $\delta\in(0,\delta_0]$ and $t\in (\tau(\delta),t_0]$, we have $\sigma(t) \geq \delta$ by (\ref{eqTau}) and the fact that $\sigma$ is non-decreasing. Since $\sigma$ is continuous on $(0,t_0]$, we get $ \sigma(\tau(\delta)) \geq \delta$. For each $x\in S^2 \setminus N_d^{\tau(\delta)} (\crit f)$, we know from the definition of $\sigma$ that $f|_{B_d(x,\sigma(\tau(\delta)))}$ is injective. Therefore $f|_{B_d(x,\delta)}$ is injective.
\end{proof}

\begin{lemma}   \label{lmTileInFlower}
Let $f$ and $\CC$ satisfy the Assumptions. Fix $m,n\in\N_0$ with $m<n$. If $f(\CC)\subseteq \CC$ and no $1$-tile in $\X^1(f,\CC)$ joins opposite sides of $\CC$, then the following statements hold:
\begin{enumerate}

\smallskip
\item[(i)] For each $n$-vertex $v\in \V^n(f,\CC)$ and each $m$-vertex $w\in\V^m(f,\CC)$, if $v\notin \overline{W}^m(w)$, then $W^m(w)\cap W^n(v)=\emptyset$.

\smallskip

\item[(ii)] For each $n$-tile $X^n\in \X^n(f,\CC)$, there exists an $m$-vertex $v^m\in \V^m(f,\CC)$ such that $X^n\subseteq W^m(v^m)$.

\smallskip
\item[(iii)] For each pair of distinct $m$-vertices $p,q\in\V^m(f,\CC)$, $\overline{W}^{n+1}(p) \cap \overline{W}^{n+1}(q) = \emptyset$.
\end{enumerate}
\end{lemma}

Recall that $W^n$ is defined in (\ref{defFlower}) and $\overline{W}^n(p)$ is the closure of $W^n(p)$. Note that a flower is an open set (see \cite[Lemma~7.2]{BM10}) and by definition a tile is a closed set. 

\begin{proof}
We first observe that in order to prove any of the statements in the lemma, it suffices to assume $n=m+1$. So we will assume, without loss of generality, that $n=m+1$.

\smallskip
(i) Since $v\notin\overline{W}^m(w)$, by (\ref{defFlower}) we get that $v\notin c$ for each $m$-cell $c\in\DD^m$ with $w\in c$. Since $f(\CC)=\CC$, for each $n$-cell $c'\in\DD^n$ and each $m$-cell $c\in\DD^m$, if $c\cap \inte(c')\neq \emptyset$, then $c'\subseteq c$ (see Lemma~4.3 and the proof of Lemma~4.7 in \cite{BM10}). Thus $c\cap \inte(c')= \emptyset$ for $c\in\DD^m$ and $c'\in\DD^n$ with $w\in c$ and $v\in c'$. So $W^m(w)\cap W^n(v)=\emptyset$ by (\ref{defFlower}).

\smallskip
(ii) Let $X^m \in \X^m$ be the unique $m$-tile with $X^n\subseteq X^m$. Depending on the location of $X^n$ in $X^m$, it suffices to prove statement~(ii) in the following cases:

\begin{enumerate}

\smallskip
\item[(1)] Assume that $X^n \subseteq \inte (X^m)$. Then $X^n \subseteq W^m(v^m)$ for any $v^m \in X^m \cap \V^m$.

\smallskip
\item[(2)] Assume that $\emptyset \neq X^n\cap e \subseteq \inte (e)$ for some $m$-edge $e\in \E^m$ with $e\subseteq X^m$. Then since no $1$-tile joins opposite sides of $\CC$, by Proposition~\ref{propCellDecomp}(i), either $X^n \cap \partial X^m \subseteq \inte (e)$ or there exists $e'\in \E^m$ such that $X^n \cap \partial X^m \subseteq \inte (e) \cup \inte (e')$ and $e\cap e' = \{v\}$ for some $v\in\V^m$. In the former case, choose any $v^m \in e\cap \V^m$; and in the latter case, let $v^m=v$. Then $X^n \subseteq W^m(v^m)$.

\smallskip
\item[(3)] Assume $X^n\cap \V^m \neq \emptyset$. Since no $1$-tile joins opposite sides of $\CC$, by Proposition~\ref{propCellDecomp}(i), there exists some $m$-vertex $v^m\in \V^m$ such that $X^n \cap \V^m =\{v^m\}$. Let $e,e'\in \E^m$ be the two $m$-edges that satisfy $e\cup e' \subseteq X^m$ and $e\cap e' = \{v^m\}$. Then by Proposition~\ref{propCellDecomp}(i) and the assumption that no $1$-tile joins opposite sides of $\CC$, we get that $X^n\cap\partial X^m \subseteq \{v^m\} \cup \inte (e) \cup \inte (e')$. Thus $X^n \subseteq W^m(v^m)$.
\end{enumerate}

\smallskip
(iii) We observe that since no $1$-tile in $\X^1$ joins opposite sides of $\CC$ and $f(\CC)\subseteq \CC$, by Proposition~\ref{propCellDecomp}(i), each $(k+1)$-tile $X^{k+1}$ contains at most one $k$-vertex, for $k\in\N_0$.

Let $p,q\in\V^m$ be distinct. Then by Remark~\ref{rmFlower} and the observation above, we know $q\notin\overline{W}^n(p)$. So by part~(i), we get $W^n(p)\cap W^{n+1}(q) = \emptyset$. Since flowers are open sets, we have $W^n(p)\cap \overline{W}^{n+1}(q) = \emptyset$. It suffices to prove that $\overline{W}^{n+1}(p) \subseteq W^n(p)$. Indeed this inclusion is true; for otherwise, there exists an $(n+1)$-tile $X^{n+1}\subseteq \overline{W}^{n+1}(p)$ and a point $x\in \overline{W}^n(p) \setminus W^n(p)$ such that $\{x,p\}\subseteq X^{n+1}$. By (\ref{defFlower}) and applying Proposition~\ref{propCellDecomp}(i), we get a contradiction to the assumption that no $1$-tile in $\X^1$ joins opposite sides of $\CC$.
\end{proof}

Let $f\: S^2\rightarrow S^2$ be an expanding Thurston map, and $\CC\subseteq S^2$ a Jordan curve containing $\post f$ such that $f(\CC) \subseteq \CC$. We denote, for $m\in\N_0$, $n\in\N$, $q\in S^2$, and $q_i\in \V^m(f,\CC)$ for $i\in\{0,1,\dots,n-1\}$,
\begin{align}    \label{eqDefEm}
  & E_m(q_0,q_1,\dots,q_{n-1};q)  \notag \\
= & \big\{ x\in f^{-n}(q) \,\big|\, f^i(x)\in \overline{W}^m(q_i), i\in \{0,1,\dots,n-1\} \big\}\\
= & f^{-n}(q) \cap \bigg( \bigcap\limits_{i=0}^{n-1} f^{-i}\(\overline{W}^m(q_i)\) \bigg),    \notag
\end{align}
where $\overline{W}^m(q_i)$ is the closure of the $m$-flower $W^m(q_i)$ as defined in Section~\ref{sctThurstonMap}.

\begin{lemma}   \label{lmCoverByFlowers}
Let $f\: S^2\rightarrow S^2$ be an expanding Thurston map, and $\CC\subseteq S^2$ a Jordan curve containing $\post f$ such that $f(\CC) \subseteq \CC$. Then 
\begin{equation}
\bigcap\limits_{i=0}^n f^{-i} (W^m(p_i))  \subseteq \bigcup\limits_{x\in E_m(p_0,p_1,\dots,p_{n-1};p_n)}  W^{m+n} (x),
\end{equation}
for $m\in\N_0$, $n\in\N$, and $p_i\in\V^m(f,\CC)$ for $i\in \{0,1,\dots,n\}$. Here $E_m$ is defined in (\ref{eqDefEm}).
\end{lemma}

\begin{proof}
We prove the lemma by induction on $n\in\N$.

For $n=1$, we know that for all $p_0,p_1\in \V^m(f,\CC)$,
\begin{align*}
W^m(p_0) \cap f^{-1}(W^m(p_1)) 
\subseteq & \bigcup \big\{ W^{m+1}(x) \,\big|\, x\in f^{-1}(p_1),x\in\overline{W}^m(p_0)  \big\}  \\
 =        & \bigcup\limits_{x\in E_m(p_0;p_1)}  W^{m+1}(x)
\end{align*}
by (\ref{eqDefEm}) and the fact that $W^{m+1}(x)\cap W^m(p_0) = \emptyset$ if both $x\in \V^{m+1}(f,\CC)$ and $x\notin \overline{W}^m(p_0)$ are satisfied (see Lemma~\ref{lmTileInFlower}(i)).

We now assume that the lemma holds for $n=l$ for some $l\in\N$.

We fix a point $p_i\in \V^m(f,\CC)$ for each $i\in \{0,1,\dots,l,l+1\}$. Then
\begin{equation*}
\bigcap\limits_{i=0}^{l+1} f^{-i} (W^m(p_i))   =  W^m(p_0) \cap f^{-1}  \bigg(\bigcap\limits_{i=1}^{l+1} f^{-(i-1)}(W^m(p_i))\bigg).
\end{equation*}
By induction hypothesis, the right-hand side of the above equation is a subset of
\begin{align*}
          & W^m(p_0) \cap  f^{-1} \bigg( \bigcup\limits_{x\in E_m(p_1,p_2,\dots,p_l;p_{l+1})} W^{m+l}(x)  \bigg) \\
  =       & \bigcup\limits_{x\in E_m(p_1,p_2,\dots,p_l;p_{l+1})} \( W^{m}(p_0) \cap f^{-1}\(W^{m+l}(x)\)  \)  \\
\subseteq & \bigcup\limits_{x\in E_m(p_1,p_2,\dots,p_l;p_{l+1})} \( \bigcup\big\{ W^{m+l+1}(y)  \,|\, y\in f^{-1}(x), y\in \overline{W}^m(p_0)\big\}  \)    \\
=         & \bigcup\limits_{x\in E_m(p_1,p_2,\dots,p_l;p_{l+1})} \bigcup\limits_{y\in E_m(p_0;x)}  W^{m+l+1}(y),
\end{align*}
where the last two lines is due to (\ref{eqDefEm}) and the fact that $W^{m+l+1}(y)\cap W^m(p_0) = \emptyset$ if both $y\in \V^{m+l+1}(f,\CC)$ and $y\notin \overline{W}^m(p_0)$ are satisfied (see Lemma~\ref{lmTileInFlower}(i)).

We claim that 
\begin{equation*}
\bigcup\limits_{x\in E_m(p_1,p_2,\dots,p_l;p_{l+1})}  E_m(p_0;x) = E_m(p_0,p_1,\dots,p_l; p_{l+1}).
\end{equation*}

Assuming the claim, we then get
\begin{equation*}
\bigcap\limits_{i=0}^{l+1} f^{-i} (W^m(p_i)) \subseteq \bigcup\limits_{x\in E_m(p_0,p_1,\dots,p_l;p_{l+1})}  W^{m+l+1}(y).
\end{equation*}

Thus it suffices to prove the claim now. Indeed, by (\ref{eqDefEm}),
\begin{align*}
   & \bigcup\limits_{x\in E_m(p_1,p_2,\dots,p_l;p_{l+1})}  E_m(p_0;x)  \\
=  & \bigg\{ y\in f^{-1}(x) \,\bigg|\, y\in \overline{W}^m(p_0), x\in f^{-l}(p_{l+1}) \cap \bigg( \bigcap\limits_{i=1}^l f^{-i+1} \( \overline{W}^m(p_i) \)  \bigg)  \bigg\} \\
=  & \bigg\{  y\in f^{-l-1}(p_{l+1}) \,\bigg|\,  y\in \overline{W}^m(p_0), f(y) \in  \bigcap\limits_{i=1}^l f^{-i+1} \( \overline{W}^m(p_i) \)   \bigg\}   \\
=  & E_m(p_0,p_1,\dots,p_l;p_{l+1}).
\end{align*}

The induction step is now complete.
\end{proof}

\smallskip

We now review the notions of a simple directed graph and of a finite rooted tree that will be used in the proof of Theorem~\ref{thmAsympHExpWOPC}. Since the only purpose of such notions is to make the statements and proofs precise, and we will not use any nontrivial facts from graph theory, we adopt here a simplified approach to define relevant concepts as quickly as possible (compare \cite{BJG09}).

A \defn{simple directed graph} $\G=(\VV(\G),\EE(\G))$ is made up from a \defn{set of vertices} $\VV(\G)$ and a \defn{set of directed edges} 
$$
\EE(\G) \subseteq \VV(\G)\times \VV(\G) \setminus \{(v,v)\,|\,v\in\VV(\G)\}.
$$
A simple directed graph $\G$ is \defn{finite} if $\card\VV(\G) < +\infty$. Two \defn{vertices} $v,w\in \VV(\G)$ are \defn{connected by a directed edge $(v,w)$} if $(v,w)\in\EE(\G)$. If $e=(v,w)\in\EE(\G)$, then we call $v$ the \defn{initial vertex} of $e$, denoted by $i(e)$, and $w$ the \defn{terminal vertex} of $e$, denoted by $t(e)$. The \defn{indegree} of a vertex $v\in\VV(\G)$ is $d^-(v)=\card\{w\in\VV(\G) \,|\, (w,v)\in \EE(\G) \}$, and the \defn{outdegree} of $v$ is $d^+(v)=\card\{w\in\VV(\G) \,|\, (v,w)\in \EE(\G) \}$. A \defn{path} from a vertex $v\in\VV(\G)$ to a vertex $w\in\VV(\G)$ is a finite sequence of vertices $v=v_0,v_1,v_2,\dots,v_{n-1},v_n=w$ such that $(v_i,v_{i+1})\in\EE(\G)$ for each $i\in\{0,1,\dots,n-1\}$. The \defn{length} of such a path is $n$. The \defn{distance from $v$ to $w$} is the minimal length of all paths from $v$ to $w$. By convention, the distance from $v$ to $v$ is $0$, and if there is no path from $v$ to $w$ for $v\neq w$, then the distance from $v$ to $w$ is $\infty$. If the distance of $v$ to $w$ is $n\in\N_0$, then we say that $w$ is at a distance $n$ from $v$.

A finite simple directed graph $\T$ is a \defn{a finite rooted tree} if there exists a vertex $r\in\VV(\T)$ such that for each vertex $v\in\VV(\T)\setminus \{r\}$ there exists a unique path from $r$ to $v$. We call such a simple directed graph a \defn{finite rooted tree with root $r$}, and $r$ the \defn{root} of $\T$. Note that a finite rooted tree has a unique root. A vertex $v$ of a finite rooted tree $\T$ is called a \defn{leaf} (of $\T$) if $d^+(v)=0$. If $(v,w)\in \EE(\T)$, then $w$ is said to be a \defn{child} of $v$.

\begin{lemma}[A bound for the number of leaves]   \label{lmTree}
Let $\T$ be a finite rooted tree with root $r$ whose leaves are all at the same distance from $r$. Assume that there exist constants $c,k\in\N$ with the following properties:
\begin{enumerate}
\smallskip
\item[(i)] $d^+(x) \leq c$ for each vertex $x\in\VV(\T)$,

\smallskip
\item[(ii)] for each leaf $v$, the number of vertices $w$ with $d^+(w) \geq 2$ in the path from $r$ to $v$ is at most $k$.
\end{enumerate}
Then then number of leaves of $\T$ is at most $c^k$.
\end{lemma}

\begin{proof}
Let $N\in\N_0$ be the distance from $r$ to any leaf of $\T$. For each $n\in\N_0$, we define $\VV_n$ as the set of vertices of $\T$ at distance $n$ from $r$. It is clear that a vertex $v\in\VV(\T)$ is a leaf of $\T$ if and only if $v\in \VV_N$.

We can recursively construct a function $h\: \VV(\T) \rightarrow \RR$ by setting $h(r)=1$, and for each $v\in \VV(\T)$, defining $h(v)= \frac{h(w)}{d^+(w)}$, where $w\in\VV(\T)$ is the unique vertex with $(w,v)\in\EE(\T)$. See Figure~\ref{figTreeSplit}.

\begin{figure}
    \centering
    \begin{overpic}
    [width=6cm, 
    tics=20]{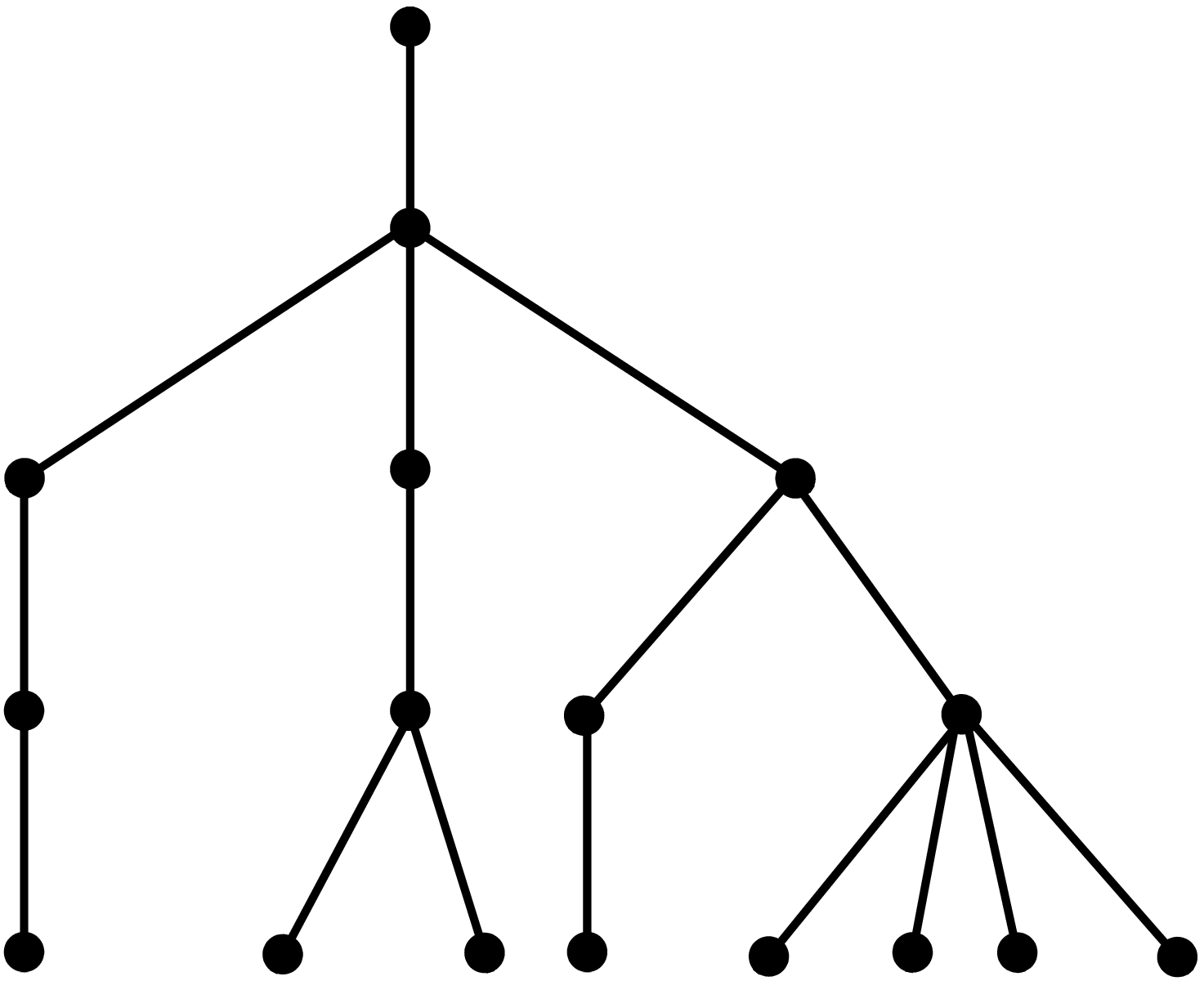}
    \put(64,132){$1$}
    \put(64,110){$1$}
    \put(-7,69){$\frac{1}{3}$}
    \put(-7,35){$\frac{1}{3}$}
    \put(0,-10){$\frac{1}{3}$}
    \put(48,69){$\frac{1}{3}$}
    \put(48,35){$\frac{1}{3}$}
    \put(37,-10){$\frac{1}{6}$}
    \put(66,-10){$\frac{1}{6}$}
    \put(117,69){$\frac{1}{3}$}
    \put(80,-10){$\frac{1}{6}$}
    \put(73,35){$\frac{1}{6}$}
    \put(141,36){$\frac{1}{6}$}
    \put(104,-10){$\frac{1}{24}$}
    \put(124,-10){$\frac{1}{24}$}    
    \put(139,-10){$\frac{1}{24}$}
    \put(162,-10){$\frac{1}{24}$}
    \end{overpic}
    \caption{The function $h$ for a finite rooted tree.}
    \label{figTreeSplit}
\end{figure}

By the two properties in the hypothesis, we have $h(v) \geq c^{-k}$ for each leaf $v \in \VV(\T)$ of $\T$. On the other hand, it is easy to see from induction that $\sum\limits_{w\in \VV_n} h(w) = 1$ for each $n\in\{0,1,\dots,N\}$. In particular, we have $\sum\limits_{w\in \VV_N} h(w) = 1$. Thus $\card \VV_N \leq c^k$. Therefore, the number of leaves of $\T$ is at most $c^k$.
\end{proof}

\subsection{Proof of Theorem~\ref{thmWeakExpansion}}   \label{subsctProofWeakExp}

We split Theorem~\ref{thmWeakExpansion} into three parts and prove each one separately here.

\begin{theorem}    \label{thmAsympHExpWOPC}
An expanding Thurston map $f\: S^2 \rightarrow S^2$ with no periodic critical points is asymptotically $h$-expansive.
\end{theorem}

\begin{proof}
We need to show $h^*(f)=0$. By (\ref{eqH*gn=nH*g}), it suffices to prove that $f^i$ is asymptotically $h$-expansive for some $i\in\N$. Note that by (\ref{eqLocalDegreeProduct}), $f^i$ has no periodic critical points for each $i\in\N$ if $f$ does not. Thus by Lemma~\ref{lmCexistsL}, we can assume, without loss of generality, that there exists a Jordan curve $\CC\subseteq S^2$ containing $\post f$ such that $f(\CC)\subseteq\CC$, and no $1$-tile joins opposite sides of $\CC$. We consider the cell decompositions of $S^2$ induced by $f$ and $\CC$ in this proof.

Recall that $\W^i$ defined in (\ref{defSetNFlower}) denotes the set of all $i$-flowers $W^i(p)$, $p\in\V^i$, for each $i\in\N_0$.

Since $f$ is expanding, it is easy to see from Lemma~\ref{lmCellBoundsBM}, Proposition~\ref{propCellDecomp}, and the Lebesgue Number Lemma (\cite[Lemma~27.5]{Mu00}) that $\{ \W^i \}_{i\in\N_0}$ forms a refining sequence of open covers of $S^2$ (see Definition~\ref{defRefSeqOpenCover}). Thus it suffices to prove that 
\begin{align}   \label{eqTailEntropy=0}
 & h^*(f) = \\
 &  \lim\limits_{m\to+\infty}  \lim\limits_{l\to+\infty} \lim\limits_{n\to+\infty}  \frac{1}{n} H \( \bigvee\limits_{i=0}^{n-1} f^{-i}\(\W^l\) \Bigg| \bigvee\limits_{j=0}^{n-1} f^{-j}\(\W^m\)  \) = 0. \notag
\end{align}
See (\ref{eqH(xi|eta)}) for the definition of $H$.

\smallskip

We now fix arbitrary $n,m,l\in N$ that satisfy $m+n>l>m$.

The plan for the proof is the following. We will first obtain an upper bound for the number of $(m+n-1)$-flowers needed to cover each element $A$ in the cover $\bigvee\limits_{j=0}^{n-1} f^{-j}\(\W^m\)$ of $S^2$. By Lemma~\ref{lmCoverByFlowers}, it suffices to find an upper bound for $\card E_m(p_0,p_1,\dots,p_{n-2};p_{n-1})$ for $p_0,p_1,\dots,p_{n-1}\in\V^m$. We identify $E_m(p_0,p_1,\dots,p_{n-2};p_{n-1})$ with the set of leaves of a certain rooted tree. By Lemma~\ref{lmTree}, we will only need to bound the number of vertices with more than one child in each path connecting the root with some leave. This can be achieved after one observes that for an expanding Thurston map with no periodic critical points, the frequency for an orbit getting near the set of critical points is bounded from above. After this main step, we will then find an upper bound for the number of $(l+n)$-tiles needed to cover $A$. By observing that each $(l+n)$-tile is a subset of some element in $\bigvee\limits_{j=0}^{n-1} f^{-j}\(\W^l\)$, we will finally obtain a suitable upper bound for $H \( \bigvee\limits_{i=0}^{n-1} f^{-i}\(\W^l\) \Bigg| \bigvee\limits_{j=0}^{n-1} f^{-j}\(\W^m\)  \)$ which leads to (\ref{eqTailEntropy=0}).

\smallskip

Let $A\in \bigvee\limits_{j=0}^{n-1} f^{-j}(\W^m)$, say
\begin{equation}   \label{eqA_WOPC}
A= \bigcap\limits_{i=0}^{n-1}  f^{-i}(W^m(p_i))
\end{equation}
where $p_0,p_1,\dots,p_{n-1} \in \V^m$. By Lemma~\ref{lmCoverByFlowers},
\begin{equation}
A \subseteq  \bigcup\limits_{x\in E_m(p_0,p_1,\dots,p_{n-2};p_{n-1})}  W^{m+n-1}(x),
\end{equation}
where $E_m$ is defined in (\ref{eqDefEm}).

We can construct a rooted tree $\T$ from $E_m(p_0,p_1,\dots,p_{n-2}; p_{n-1})$ as a simple directed graph. The set $\VV(\T)$ of vertices of  $\T$ is
\begin{equation*}
\VV(\T) = \bigcup\limits_{i=0}^{n-1}  \big\{ (f^i(x), n-1-i) \in S^2\times \N_0 \,\big|\, x\in E_m(p_0,p_1,\dots,p_{n-2}; p_{n-1})  \big\}.
\end{equation*}
Two vertices $(x,i),(y,j)\in \VV(\T)$ are connected by a directed edge $((x,i),(y,j)) \in \EE(\V)$ if and only if $f(y)=x$ and $j=i+1$. Clearly the simple directed graph $\T$ constructed this way is a finite rooted tree with root $(p_{n-1},0)\in\VV(\T)$.

Observe that if a vertex $(x,i)\in \VV(\T)$ is a leaf of $\T$, then $x\in f^{-n+1}(p_{n-1})$ and $i=n-1$.

\begin{figure} 
    \centering
    \begin{overpic}
    [width=6cm, 
    tics=20]{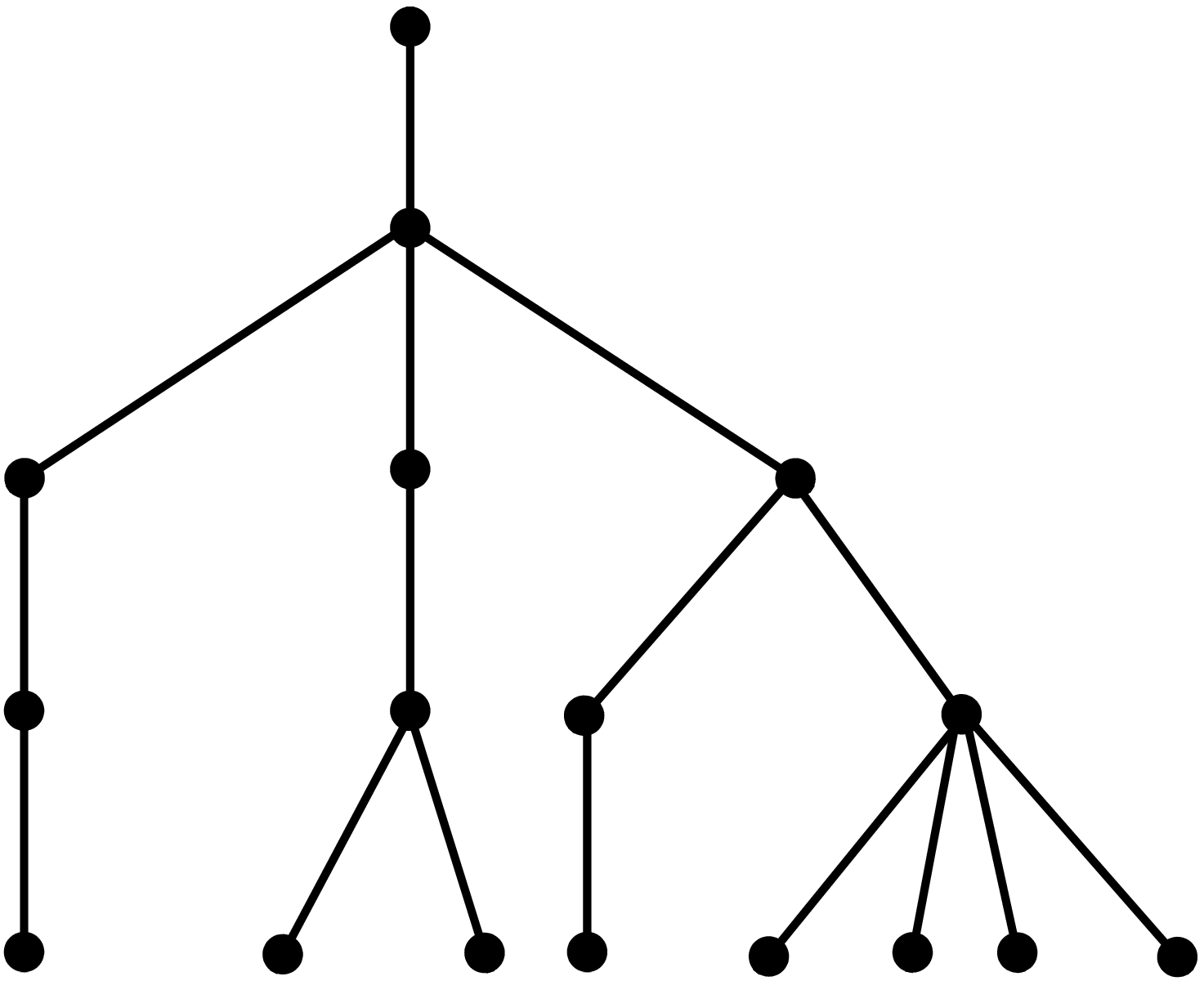}
    \put(64,132){$(p_{n-1},0)$}
    \put(141,36){$(v,3)$}
    \end{overpic}
    \caption{An example of $\T$ with $n=5$ and $c(v,3)=4$.}
    \label{figTree_WOPC}
\end{figure}

For each $(x,i)\in\VV(\T)$, we write $c(x,i) = d^+((x,i))$, i.e.,
\begin{equation}   \label{eqDefC_NoChildren_WOPC}
c(x,i) = \card \{ (y,i+1) \in\VV(\T) \,|\, f(y)=x \}.
\end{equation}
We make the convention that for each $x\in S^2$ and each $i\in\Z$, if $(x,i)\notin \VV(\T)$, then $c(x,i)=-1$. See Figure~\ref{figTree_WOPC} for an example of $\T$.

Recall that by (\ref{eqDefEm}),
\begin{align*}
    & E_m(p_0,p_1,\dots,p_{n-2};p_{n-1}) \\
=   & \big\{ y\in f^{-n+1}(p_{n-1}) \,\big|\, f^i(y) \in \overline{W}^m(p_i),i\in\{0,1,\dots,n-2\}  \big\}.
\end{align*}
So if $(x,i)\in\VV(\T)$, then $c(x,i)$ is at most the number of distinct preimages of $x$ under $f$ contained in $\overline{W}^m(p_{i+1})$. Thus
\begin{equation}  \label{eqBoundForNoChildren_WOPC}
0\leq c(x,i) \leq \deg f  \text{ for } (x,i)\in\VV(\T).
\end{equation}

Fix a visual metric $d$ on $S^2$ for $f$ with expansion factor $\Lambda >1$. The map $f$ is Lipschitz with respect to $d$ (see Lemma~\ref{lmLipschitz}). Then there exists a constant $K\geq 1$ depending only on $f$ and $d$ such that $d(f(x),f(y)) \leq K d(x,y)$ for $x,y\in S^2$. We may assume that $K\geq 2$.

Define
$$
N_c=\max \{\min \{i\in\N \,|\, f^j(x) \notin \crit f \text{ if } j\geq i\} \,|\, x\in \crit f \}.
$$
The maximum is taken over a finite set of integers since $f$ has no periodic critical points. So $N_c\in\N$. Note that by definition, if $x\in\crit f$, then $f^i(x) \in \post f\setminus \crit f$ for each $i\geq N_c$. Denote the shortest distance between a critical point and the set $\post f\setminus \crit f$ by
$$
D_c=\min \{d(x,y) \,|\, x\in \post f\setminus \crit f, y\in \crit f \}.
$$
Then $D_c \in (0,+\infty)$ since both $\post f \setminus \crit f$ and $\crit f$ are nonempty finite sets.

\smallskip

We now proceed to find an upper bound for 
\begin{equation*}
\card \big\{ i\in \{0,1,\dots,n-1\} \,\big|\, c(f^i(z),n-1-i) \geq 2  \big\}
\end{equation*}
for each $(z,n-1)\in\VV(\T)$, uniform in $(z,n-1)$. Recall that $z\in f^{-n+1} (p_{n-1})$ for each $(z,n-1)\in\VV(\T)$. We fix such a point $z$.

\begin{figure}
    \centering
    \begin{overpic}
    [width=12cm, 
    tics=20]{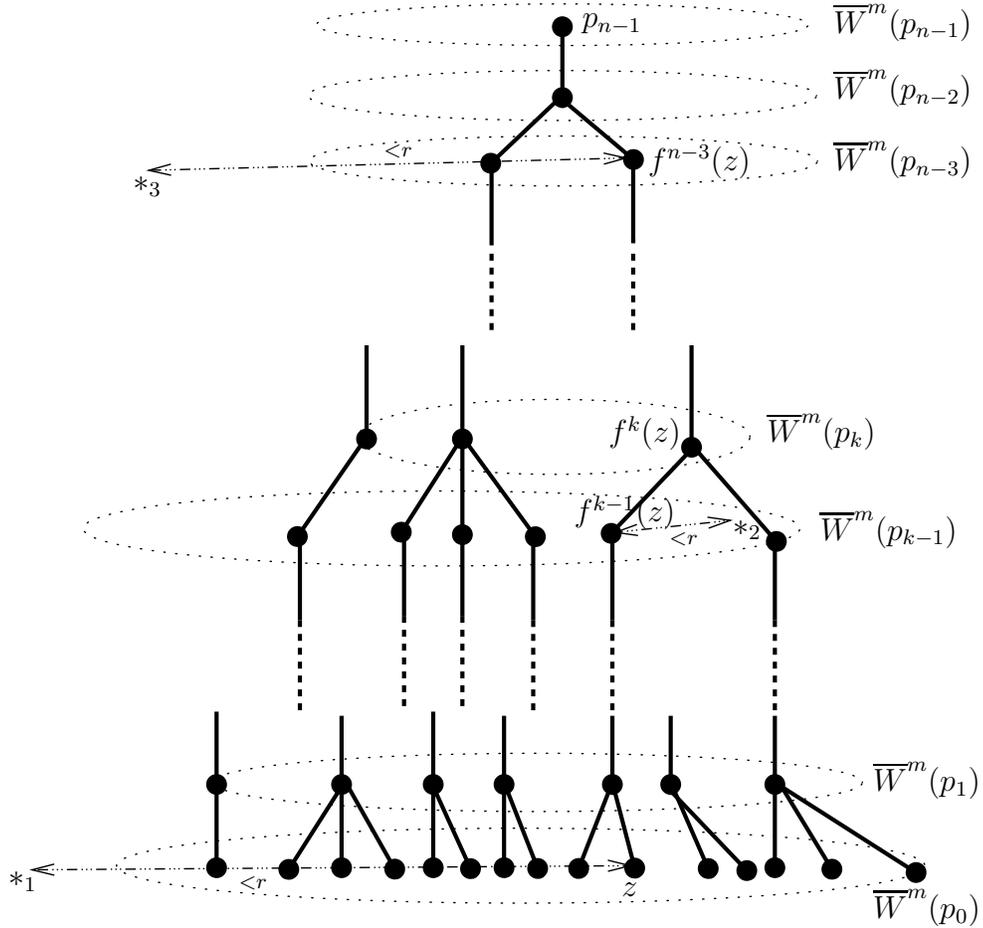}
    \put(210,332){$p_{n-1}$}
    \put(305,330){$\overline{W}^m(p_{n-1})$}
    \put(305,305){$\overline{W}^m(p_{n-2})$}
    \put(305,278){$\overline{W}^m(p_{n-3})$}
    \put(280,175){$\overline{W}^m(p_{k})$}
    \put(300,138){$\overline{W}^m(p_{k-1})$}
    \put(320,43){$\overline{W}^m(p_{1})$}
    \put(320,-5){$\overline{W}^m(p_{0})$} 
    \put(235,278){$f^{n-3}(z)$}  
    \put(220,175){$f^k(z)$}
    \put(207,145){$f^{k-1}(z)$}
    \put(225,3){$z$} 
    \put(40,270){$*_3$}
    \put(267,140){$*_2$}
    \put(-7,7){$*_1$}
    \put(80,8){$_{<r}$}
    \put(243,138){$_{<r}$}
    \put(135,285){$_{<r}$}
    \end{overpic}
    \caption{$*_1,*_2,*_3\in\crit f$, $r=\tau(3C\Lambda^{-m})$, and $c(f^k(z),n-1-i)=2$.}
    \label{figLocInjectAwayFromCrit_WOPC}
\end{figure}

In order to find an upper bound, we first define, for each $i\in\N$ sufficiently large,
\begin{equation}   \label{eqDefMm_WOPC}
M_i = \bigg\lfloor \log_K \(\frac{D_c - \tau(3C\Lambda^{-i})}{\tau(3C\Lambda^{-i})}  \)   \bigg\rfloor  -2,
\end{equation}
where the function $\tau$ is from Lemma~\ref{lmLocInjectAwayFromCrit}, and $C\geq 1$ is a constant depending only on $f$, $\CC$, and $d$ from Lemma~\ref{lmCellBoundsBM}. Note that $\tau(3C\Lambda^{-i})\longrightarrow 0$ as $i\longrightarrow+\infty$ (Lemma~\ref{lmLocInjectAwayFromCrit}), thus $M_i$ is well-defined for $i$ sufficiently large, and
\begin{equation}    \label{eqM_mToInfty_WOPC}
\lim\limits_{i\to+\infty} M_i = +\infty.
\end{equation}

We assume that $m$ is sufficiently large such that the following conditions are both satisfied:
\begin{enumerate}
\smallskip
\item[(i)] $m > \log_\Lambda \big(\frac{3C}{\delta_0}\big)$,

\smallskip
\item[(ii)] $M_m > N_c$,
\end{enumerate}
where $\delta_0\in(0,1]$ is a constant that depends only on $f$ and $d$ from Lemma~\ref{lmLocInjectAwayFromCrit}. Note that by Lemma~\ref{lmCellBoundsBM}, each $m$-flower is of diameter at most $2C\Lambda^{-m}$. Thus condition~(i) implies that for each $v\in\V^m$, each pair of points $x,y \in \overline{W}^m(v)$ satisfy $d(x,y)< 3C\Lambda^{-m} < \delta_0$.

Fix $k\in\{0,1,\dots,n-1\}$ with $c\(f^k(z),n-1-k\) \geq 2$. Then $k\neq 0$ and the number of distinct points in $\overline{W}^m(p_{k-1})$ that are mapped to $f^k(z)$ under $f$ is at least $c\(f^k(z),n-1-k\) \geq 2$. Thus $f$ is not injective on $\overline{W}^m(p_{k-1})$. See Figure~\ref{figLocInjectAwayFromCrit_WOPC}. By Lemma~\ref{lmCellBoundsBM}, $\diam_d\(\overline{W}^m(p_{k-1})\)\leq 2 C\Lambda^{-m}$. Since $f^{k-1}(z)\in \overline{W}^m(p_{k-1})$, the map $f$ is not injective on $B_d\(f^{k-1}(z),3C\Lambda^{-m}\)$. Then since $3C\Lambda^{-m}<\delta_0$, by Lemma~\ref{lmLocInjectAwayFromCrit},
\begin{equation*}
d\(f^{k-1}(z),\crit f\) < \tau\(3C\Lambda^{-m}\).
\end{equation*}
Choose $w\in\crit f$ that satisfies
$
d\(f^{k-1}(z),w\) < \tau\(3C\Lambda^{-m}\).
$
Then for each $j\in \N_0$,
\begin{equation}     \label{eqPfThmAsympHExp1_WOPC}
d\(f^{k+j-1}(z),f^j(w) \) < K^j \tau\(3C\Lambda^{-m}\).
\end{equation}

We will show that in the sequence $f^k(z),f^{k+1}(z),\dots,f^{k+M_m}(z)$, the number of terms $f^{k+j}(z)$, $0\leq j\leq M_m$, for which the vertex 
$$
\(f^{k+j}(z), n-1-k-j\) \in \VV(\T)
$$
has at least two children is bounded above by $N_c$, i.e.,
\begin{equation}     \label{eqNcEveryMm_WOPC}
\card \big\{ j\in\{0,1,\dots,M_m\}  \,\big|\, c\(f^{k+j}(z), n-1-k-j\)\geq 2 \big\} \leq N_c.
\end{equation}
Note that $M_m$ is defined in (\ref{eqDefMm_WOPC}). Here we use the convention that for each $x\in S^2$ and each $i\in \Z$, if $(x,i)\notin \VV(\T)$, then $c(x,i)=-1$.

Indeed, for each $j\in\{N_c,N_c +1,\dots,\min\{M_m,n-1-k\}\}$, we have $f^j(w)\in \post f\setminus \crit f$. Note that here $M_m>N_c$ by condition~(ii) on $m$. Thus by (\ref{eqPfThmAsympHExp1_WOPC}) and (\ref{eqDefMm_WOPC}),
\begin{align*}
d\(f^{k+j-1}(z),\crit f \)  & \geq d\(\crit f, f^j(w)\) - d\( f^j(w), f^{k+j-1}(z) \)\\
                            & \geq D_c - K^j \tau( 3C\Lambda^{-m})   \\
                            & \geq D_c - K^{M_m} \tau( 3C\Lambda^{-m})  \\
                            & \geq D_c - \( \frac{D_c - \tau(3C\Lambda^{-m})}{\tau(3C\Lambda^{-m})} \) \tau(3C\Lambda^{-m} )\\
                            & =  \tau(3C\Lambda^{-m} ).
\end{align*}
Hence by Lemma~\ref{lmLocInjectAwayFromCrit}, the restriction of $f$ to $B_d\(f^{k+j-1}(z), 3C\Lambda^{-m}\)$ is injective. Note that $f^{k+j-1}(z) \in \overline{W}^m(p_{k+j-1})$, and by Lemma~\ref{lmCellBoundsBM}, $\diam_d \(\overline{W}^m(p_{k+j-1}) \) \leq 2C\Lambda^{-m}$. So $f$ is injective on $\overline{W}^m(p_{k+j-1})$. Thus 
$$
c\(f^{k+j}(z), n-1-k-j\) = 1
$$
for each $j\in\{N_c,N_c +1,\dots,\min\{M_m,n-1-k\}\}$. Hence 
$$
c\(f^{k+j}(z), n-1-i-j\) \in \{1,-1\}
$$
for each $j\in\{N_c,N_c +1,\dots,M_m\}$. Then (\ref{eqNcEveryMm_WOPC}) holds.

Thus we get that
\begin{equation}  \label{eqSplitNoBound_WOPC}
\card \big\{ i\in \{0,1,\dots,n-1\} \,\big|\, c(f^i(z),n-1-i) \geq 2  \big\}  \leq N_c \bigg\lceil \frac{n}{M_m} \bigg\rceil
\end{equation}
for each $(z,n-1)\in\VV(\T)$.

Hence by (\ref{eqSplitNoBound_WOPC}), (\ref{eqBoundForNoChildren_WOPC}), and Lemma~\ref{lmTree}, we can conclude that the number of leaves of $\T$ is at most $(\deg f)^{N_c  \(\frac{n}{M_m} +1\)}$, or equivalently,
\begin{equation}
\card E_m(p_0,p_1,\dots,p_{n-2};p_{n-1}) \leq (\deg f)^{N_c  \(\frac{n}{M_m} +1\)}.
\end{equation}

We have obtained an upper bound for the number of $(m+n-1)$-flowers needed to cover $A$. Next, we will find an upper bound for the number of $(m+n-1)$-tiles, and consequently, an upper bound for the number of $(l+n)$-tiles, needed to cover $A$.

Denote the maximum number of $i$-tiles contained in the closure of any $i$-flower, over all $i\in\N_0$, by $W_f$, i.e.,
\begin{equation*}
W_f = \sup \big\{ \card \big\{ X^i\in\X^i \,\big|\, X^i\subseteq \overline{W}^j(v) \big\} \,\big|\, j\in\N_0,v\in\V^j  \big\}.
\end{equation*}
Observe that $W_f = \sup \{ 2 \deg_{f^i}(v) \,|\, i\in\N_0, \, v\in\V^i   \}$. Since $f$ has no periodic critical points, it follows from \cite[Lemma~17.1]{BM10} that $W_f$ is a finite number that only depends on $f$.

Thus we can cover $A$ in (\ref{eqA_WOPC}) by a collection of $(m+n-1)$-tiles of cardinality at most $W_f (\deg f)^{N_c  \(\frac{n}{M_m} +1\)}$.

\smallskip
On the other hand, we claim that each $(l+n)$-tile $X^{l+n}\in \X^{l+n}$ is a subset of at least one element in the open cover $\bigvee\limits_{i=0}^{n-1} f^{-i} (\W^l)$ of $S^2$. To prove the claim, we first fix an $(l+n)$-tile $X^{l+n} \in \X^{l+n}$. By Proposition~\ref{propCellDecomp}(ii) and Lemma~\ref{lmTileInFlower}(ii), for each $i\in\{0,1,\dots,n-1\}$, there exists an $l$-vertex $v_i\in \V^l$ such that $f^i\(X^{l+n}\) \subseteq W^l(v_i)$. Thus
$$
X^{l+n} \subseteq \bigcap\limits_{i=0}^{n-1} f^{-i} \( W^l(v_i)\).
$$
The proof for the claim is complete.

\smallskip

Note that for each $(m+n-1)$-tile $X^{m+n-1}\in \X^{m+n-1}$, the collection
$$
\big\{  X^{l+n} \in \X^{l+n} \,\big|\, X^{l+n}\subseteq X^{m+n-1} \big\}
$$
forms a cover of $X^{m+n-1}$, and has cardinality at most $(2\deg f)^{l-m+1}$, which follows immediately from Proposition~\ref{propCellDecomp}.  

Hence, we get that for each element $A$ of $\bigvee\limits_{j=0}^{n-1} f^{-j} (\W^m)$, we can find a cover of $A$ consisting of elements of $\bigvee\limits_{i=0}^{n-1} f^{-i} (\W^l)$ in such a way that the cardinality of the cover is at most $(2\deg f)^{l-m+1} W_f (\deg f)^{N_c  \(\frac{n}{M_m} +1\)}$. 

We conclude that
\begin{align*}
h^*(f) & = \lim\limits_{m\to+\infty} \lim\limits_{l\to+\infty} \lim\limits_{n\to+\infty} \frac{1}{n} \log \( (2\deg f)^{l-m+1} W_f (\deg f)^{N_c  \(\frac{n}{M_m} +1\)} \)  \\
       & = \lim\limits_{m\to+\infty} \lim\limits_{l\to+\infty} \lim\limits_{n\to+\infty} \frac{1}{n} N_c  \(\frac{n}{M_m} +1\) \log(\deg f) \\
       & = \lim\limits_{m\to+\infty} \frac{N_c \log(\deg f)}{M_m} \\
       & = 0.
\end{align*}
The last equality follows from (\ref{eqM_mToInfty_WOPC}). 
\end{proof}

Recall that a point $x\in S^2$ is a periodic point of $f\: S^2\rightarrow S^2$ with period $n$ if $f^n(x)=x$ and $f^i(x)\neq x$ for each $i\in\{1,2,\dots,n-1\}$.

\begin{theorem}    \label{thmNotAsympHExp}
An expanding Thurston map $f\: S^2\rightarrow S^2$ with at least one periodic critical point is not asymptotically $h$-expansive.
\end{theorem}

\begin{proof}
We need to show $h^*(f)>0$. By (\ref{eqH*gn=nH*g}), it suffices to prove that $f^i$ is not asymptotically $h$-expansive for some $i\in\N$. Note that by (\ref{eqLocalDegreeProduct}), if a point $x\in S^2$ is a periodic critical point of $f^i$ for some $i\in\N$, then it is a periodic point of $f$ and there exists $j\in\N_0$ such that $f^j(x)$ is a periodic critical point of $f$. Thus each periodic critical point of $f^\tau$ is a fixed point of $f^\tau$ if $\tau\in\N$ is a common multiple of the periods of all the periodic critical points of $f$. Hence by Lemma~\ref{lmCexistsL}, we can assume, without loss of generality, that there exists a Jordan curve $\CC\subseteq S^2$ containing $\post f$ such that $f(\CC)\subseteq\CC$, and no $1$-tile joins opposite sides of $\CC$, and each periodic critical point of $f$ is a fixed point of $f$.

Let $p$ be a critical point of $f$ that is fixed by $f$. 

In addition, we can assume, without loss of generality, that $f^{-1}(p)\setminus \CC \neq \emptyset$. Indeed, by Lemma~\ref{lmPreImageDense}, there exists $j\in\N$ such that $f^{-j}(p)\setminus \CC \neq \emptyset$. We replace $f$ by $f^j$, and observe that by (\ref{eqLocalDegreeProduct}) and the fact that each periodic critical point of $f$ is a fixed point of $f$, the set of periodic critical points of $f$ and that of $f^j$ coincide. Note that for the new map and its invariant curve $\CC$, no $1$-tile joins opposite sides of $\CC$, and each periodic critical point is a fixed point.

From now on, we consider the cell decompositions of $S^2$ induced by $f$ and $\CC$ in this proof.

\smallskip

Recall that for $i\in\N_0$, we denote by $\W^i$ as in (\ref{defSetNFlower}) the set of all $i$-flowers $W^i(p)$ where $p\in\V^i$.

Since $f$ is expanding, it is easy to see from Lemma~\ref{lmCellBoundsBM}, Proposition~\ref{propCellDecomp}, and the Lebesgue Number Lemma (\cite[Lemma~27.5]{Mu00}) that $\{ \W^i \}_{i\in\N_0}$ forms a refining sequence of open covers of $S^2$ (see Definition~\ref{defRefSeqOpenCover}). Thus it suffices to prove that 
\begin{equation*}
h^*(f) = \lim\limits_{m\to+\infty}  \lim\limits_{l\to+\infty} \lim\limits_{n\to+\infty}  \frac{1}{n} H \( \bigvee\limits_{i=0}^{n-1} f^{-i}\(\W^l\) \Bigg| \bigvee\limits_{j=0}^{n-1} f^{-j}\(\W^m\)  \) > 0.
\end{equation*}
See (\ref{eqH(xi|eta)}) for the definition of $H$.

Our plan is to construct a sequence $\{v_i\}_{i\in\N}$ of $m$-vertices such that for each $n\in\N$, the number of elements in $\bigvee\limits_{i=0}^{n-1} f^{-i}\(\W^l\)$ needed to cover $B_n=\bigcap\limits_{j=0}^{n-1} f^{-j} (W^m(v_{n-j}))$ can be bounded from below in such a way that $h^*(f)>0$ follows immediately. More precisely, we observe that the more connected components $B_n$ has, the harder to cover $B_n$. So we will choose $\{v_i\}_{i\in\N}$ as a periodic sequence of $m$-vertices shadowing an infinite backward pseudo-orbit under iterations of $f$ in such a way that each period of $\{v_i\}_{i\in\N}$ begins with a backward orbit starting at $p$ and approaching $p$ as the index $i$ increases, and then ends with a constant sequence staying at $p$. By a recursive construction, we keep track of each $B_n$ by a finite subset $V_n\subseteq B_n$ with the property that $\card(A\cap V_n)\leq 1$ for each $A \in \bigvee\limits_{i=0}^{n-1} f^{-i}\(\W^l\)$. A quantitative control of the size of $V_n$ leads to the conclusion that $h^*(f)>0$. The fact that the constant part of each period of $\{v_i\}_{i\in\N}$ can be made arbitrarily long is essential here and is not true if $f$ has no periodic critical points.

\smallskip

For this we fix $m,l\in\N$ with $l>m+100$.

Let $k=\deg_f(p)$. Then $k>1$.

Define $q_0 = p$ and choose $q_1\in f^{-1}(p)\setminus \CC$. Then $q_1$ is necessarily a $1$-vertex, but not a $0$-vertex, i.e., $q_1\in\V^1\setminus \V^0$. Since $q_1\notin \CC$, we have $q_1\in W^0(p)$. By (\ref{defFlower}), the only $2$-vertex contained in $W^2(p)$ is $p$. So $q_1\in W^0(p)\setminus W^2(p)$. Since $f\(W^i(p)\) = W^{i-1}(p)$ for each $i\in\N$ (see Remark~\ref{rmFlower}), we can recursively choose  $q_j\in\V^j$ for $j\in\{2,3,\dots,m\}$ such that
\begin{enumerate}
\smallskip
\item[(i)] $f(q_j) = q_{j-1}$,

\smallskip
\item[(ii)] $q_j\in W^{j-1}(p) \setminus W^{j+1}(p)$.
\end{enumerate}
We define a singleton set $Q_j = \{ q_j \}$.

We set $q_m^1=q_m$.

Next, we choose recursively, for each $j\in\{m+1,m+2,\dots, l-2\}$, a set $Q_j$ with $\card Q_j = k^{j-m}$ consisting of distinct points $q_j^i\in\V^j$, $i\in\{1,2,\dots,k^{j-m}\}$, such that
\begin{enumerate}
\smallskip
\item[(i)] $f(Q_j)= Q_{j-1}$,

\smallskip 
\item[(ii)] $Q_j \subseteq W^{j-1}(p)\setminus W^{j+1}(p)$.
\end{enumerate}
Note by Remark~\ref{rmFlower}, it is clear that these two properties uniquely determines $Q_j$ from $Q_{j-1}$.

Finally, we construct recursively, for $j\in\{l-1,l,l+1\}$, a set $Q_j$ with $\card Q_j = k^{l-2-m}$ consisting of distinct points $ q_j^i\in\V^j$, $i\in\big\{1,2,\dots, k^{l-2-m} \big\}$, such that
\begin{enumerate}
\smallskip
\item[(i)] $f(q_j^i)= q_{j-1}^i$,

\smallskip 
\item[(ii)] $Q_j \subseteq W^{j-1}(p)\setminus W^{j+1}(p)$.
\end{enumerate}

\smallskip

We will now construct recursively, for each $n=(l+1)s+r$, with $s\in\N_0$ and $r\in\{0,1,\dots,l\}$, an $m$-vertex $v_n\in\V^m$ and a set of $n$-vertices $V_n\subseteq \V^n$ such that the following properties are satisfied:
\begin{enumerate}
\smallskip
\item[(1)] $V_n\subseteq W^m(v_n)$ for $n\in \N_0$;

\smallskip
\item[(2)] $f\(V_n\) = V_{n-1}$ for $n\in\N$;

\smallskip
\item[(3)] For $s\in\N_0$, and
\begin{enumerate}

\smallskip
\item[(i)] for $r=0$, $V_{(l+1)s+r} \subseteq W^l(p)$,

\smallskip
\item[(ii)] for $r\in\{1,2,\dots, m\}$, $V_{(l+1)s+r} \subseteq W^{l+1}\(v_{(l+1)s+r}\)$,

\smallskip
\item[(iii)] for $r\in\{m+1,m+2,\dots, l-2\}$, there exists, for each $i\in\{1,2,\dots,k^{r-m}\}$, a subset $V^i_{(l+1)s+r}$ of $V_{(l+1)s+r}$ such that 
\begin{enumerate}
\smallskip
\item[(a)] $V^i_{(l+1)s+r} \cap V^j_{(l+1)s+r} = \emptyset$ for $1\leq i< j\leq k^{r-m}$,

\smallskip
\item[(b)] $\bigcup\limits_{i=1}^{k^{r-m}} V^i_{(l+1)s+r} = V_{(l+1)s+r}$,

\smallskip
\item[(c)] $V^i_{(l+1)s+r} \subseteq W^{l+1} \(q_r^i\)$,
\end{enumerate}

\smallskip
\item[(iv)] for $r\in\{l-1,l\}$, there exists, for each $i\in\{1,2,\dots,k^{l-2-m}\}$, a subset $V^i_{(l+1)s+r}$ of $V_{(l+1)s+r}$ such that 
\begin{enumerate}
\smallskip
\item[(a)] $V^i_{(l+1)s+r} \cap V^j_{(l+1)s+r} = \emptyset$ for $1\leq i< j\leq k^{l-2-m}$,

\smallskip
\item[(b)] $\bigcup\limits_{i=1}^{k^{l-2-m}} V^i_{(l+1)s+r} = V_{(l+1)s+r}$,

\smallskip
\item[(c)] $V^i_{(l+1)s+r} \subseteq W^{l+1} \(q_r^i\)$;
\end{enumerate}

\end{enumerate}

\smallskip
\item[(4)] for $n\in\N_0$, $A\in\bigvee\limits_{i=0}^{n-1} f^{-i}\(\W^l\)$, and $x,y\in V_n$ with $x\neq y$, we have $\{x,y\}\nsubseteq A$.
\end{enumerate}

\smallskip

We start our construction by first defining $v_n\in \V^m$ for each $n\in\N$. For $s\in\N_0$ and $r\in\{0,1,\dots, m\}$, set $v_{(l+1)s+r} = q_r$. For $s\in\N_0$ and $r\in\{m+1,m+2,\dots, l\}$, set $v_{(l+1)s+r} = p$.

We now define $V_n$ recursively.

Let $V_0= \{q_0\}$. Clearly $V_0$ satisfies properties~(1) through (4).

Assume that $V_n$ is defined and satisfies properties~(1) through (4) for each $n\in\{0,1,\dots, (l+1)s+r \}$, where $s\in\N_0$ and $r\in\{0,1,\dots,l\}$. We continue our construction in the following cases depending on $r$.

\smallskip
\emph{Case 1.} Assume $r\in\{0,1,\dots,m-1\}$. Then $v_{(l+1)s+r}=q_r$ and $v_{(l+1)s+r+1}=q_{r+1}$. 

Since $f\(W^{l+1}\(q_{r+1}\)\) = W^l\(q_r\)$ (see Remark~\ref{rmFlower}), and $V_{(l+1)s+r}\subseteq W^l(q_r)$ by the induction hypothesis, we can choose, for each $x\in V_{(l+1)s+r}$, a point $x'\in W^{l+1}(q_{r+1})$ such that $f(x')=x$. Then define $V_{(l+1)s+r+1}$ to be the collection of all such chosen $x'$ that corresponds to $x\in V_{(l+1)s+r}$. Note that
$$
\card V_{(l+1)s+r+1} = \card V_{(l+1)s+r}.
$$

All properties required for $V_{(l+1)s+r+1}$ in the induction step are trivial to verify. We only consider the last property here. Indeed, suppose that $x,y\in V_{(l+1)s+r+1}$ satisfy that $x\neq y$ and $\{x,y\}\subseteq A$ for some $A\in\bigvee\limits_{i=0}^{(l+1)s+r} f^{-i}\(\W^l\)$. Then by construction $f(x),f(y)$, and $f(A)$ satisfy
\begin{enumerate}
\smallskip
\item[(a)] $f(A)\subseteq B$ for some $B\in \bigvee\limits_{i=0}^{(l+1)s+r-1} f^{-i}\(\W^l\)$,

\smallskip
\item[(b)] $f(x), f(y)\in V_{(l+1)s+r}$, and $f(x)\neq f(y)$,

\smallskip
\item[(c)] $\{f(x),f(y)\} \subseteq f(A)\subseteq B$.
\end{enumerate}
This contradicts property~(4) for $V_{(l+1)s+r}$ in the induction hypothesis.

\smallskip
\emph{Case 2.} Assume $r\in\{m,m+1,\dots,l-3\}$. Then $v_{(l+1)s+r+1}=p$, $v_{(l+1)s+m} = q_m$, and when $r\neq m$, we have $v_{(l+1)s+r}=p$. 

If $r=m$, we define $V^1_{(l+1)s+r}=V_{(l+1)s+r}$. Recall that $q_m^1=q_m$.

Note that for each $i\in\{1,2,\dots,k^{r+1-m}\}$, $f\(W^{l+2}\(q_{r+1}^i\)\) = W^{l+1}\(q_r^j\)$ for some $j\in\{1,2,\dots,k^{r-m}\}$ (see Remark~\ref{rmFlower}), and $V^j_{(l+1)s+r} \subseteq W^{l+1}\(q^j_r\)$ by the induction hypothesis. For each $j\in\{1,2,\dots,k^{r-m}\}$, each $x\in V^j_{(l+1)s+r}$, and each $i\in\{1,2,\dots,k^{r+1-m}\}$ with $f\(W^{l+2}\(q_{r+1}^i\)\) = W^{l+1}\(q_r^j\)$, we can choose a point $x'\in W^{l+2}\(q_{r+1}^i\)$ such that $f(x')=x$. Then define $V^i_{(l+1)s+r+1}$ to be the collection of all such chosen $x'$ that corresponds to $x\in V^j_{(l+1)s+r}$. Set $V_{(l+1)s+r+1}= \bigcup\limits_{i=1}^{k^{r+1-m}} V^i_{(l+1)s+r+1}$. 

Since $Q_{r+1} \subseteq \V^{r+1} \cap W^m(p)$, $r\in\{m,m+1,\dots,l-3\}$, $l>m+100$, and no $1$-tile joins opposite sides of $\CC$, we get that
\begin{enumerate}
\smallskip
\item[(a)] for $i,j\in\{ 1,2,\dots,k^{r+1-m}\}$ with $i\neq j$, by Lemma~\ref{lmTileInFlower}(iii),
$$
W^{l+2}\(q^i_{r+1}\) \cap W^{l+2}\(q^j_{r+1}\) =\emptyset,
$$
and so $V^i_{(l+1)s+r+1} \cap V^j_{(l+1)s+r+1} = \emptyset$,

\smallskip
\item[(b)] $V_{(l+1)s+r+1} \subseteq W^m(p)$.
\end{enumerate}
Thus
$$
\card V_{(l+1)s+r+1} = k \card V_{(l+1)s+r}.
$$

We only need to verify property~(4) required for $V_{(l+1)s+r+1}$ in the induction step now. Indeed, suppose that $x,y\in V_{(l+1)s+r+1}$ with $x\neq y$ and $\{x,y\}\subseteq A$ for some $A\in\bigvee\limits_{a=0}^{(l+1)s+r} f^{-a}\(\W^l\)$. Then $A\subseteq W^l(v^l)$ for some $v^l\in \V^l$. By construction, there exist $i,j \in \{1,2,\dots,k^{r+1-m}\}$ such that $x\in W^{l+2}\(q^i_{r+1}\)$ and $y\in W^{l+2}\(q^j_{r+1}\)$. Note that $q^i_{r+1}, q^j_{r+1} \in \V^{r+1}$, $r\in\{m,m+1,\dots,l-3\}$, and $l>m+100$. So $q^i_{r+1}, q^j_{r+1} \in \V^{l-2}$. Since $x\in W^l(v^l)\cap W^{l+2}\(q^i_{r+1}\)$, we get $q^i_{r+1}\in \overline{W}^l(v_l)$ by Lemma~\ref{lmTileInFlower}(i), and thus $v^l\in \overline{W}^l(q^i_{r+1})$. Similarly $v^l\in \overline{W}^l(q^j_{r+1})$. Since $q^i_{r+1}, q^j_{r+1} \in \V^{l-2}$ and no $1$-tile joins opposite sides of $\CC$, we get from Lemma~\ref{lmTileInFlower}(iii) that $q^i_{r+1}=q^j_{r+1}$, i.e., $i=j$. Thus $f(x)\neq f(y)$ by construction. But then $f(x),f(y)$, and $f(A)$ satisfy
\begin{enumerate}
\smallskip
\item[(a)] $f(A)\subseteq B$ for some $B\in \bigvee\limits_{a=0}^{(l+1)s+r-1} f^{-a}\(\W^l\)$,

\smallskip
\item[(b)] $f(x), f(y)\in V_{(l+1)s+r}$, and $f(x)\neq f(y)$,

\smallskip
\item[(c)] $\{f(x),f(y)\} \subseteq f(A)\subseteq B$.
\end{enumerate}
This contradicts property~(4) for $V_{(l+1)s+r}$ in the induction hypothesis.

\smallskip
\emph{Case 3.} Assume $r\in\{l-2,l-1,l\}$, then $v_{(l+1)s+r+1}=v_{(l+1)s+r}=p$. 

Note that for each $i\in\{1,2,\dots,k^{l-2-m}\}$, $f\(W^{l+2}\(q_{r+1}^i\)\) = W^{l+1}\(q_r^i\)$ (see Remark~\ref{rmFlower}), and $V^i_{(l+1)s+r} \subseteq W^{l+1}\(q^i_r\)$ by the induction hypothesis. For each $j\in\{1,2,\dots,k^{l-2-m}\}$ and each $x\in V^i_{(l+1)s+r}$, we can choose a point $x'\in W^{l+2}\(q_{r+1}^i\)$ such that $f(x')=x$. Then define $V^i_{(l+1)s+r+1}$ to be the collection of all such chosen $x'$ that corresponds to $x\in V^i_{(l+1)s+r}$. Set $V_{(l+1)s+r+1}= \bigcup\limits_{i=1}^{k^{l-2-m}} V^i_{(l+1)s+r+1}$. 

Since $Q_{r+1} \subseteq \V^{r+1} \cap W^r(p)$, $r\in\{l-2,l-1,l\}$, and $l>m+100$, we get that
\begin{enumerate}
\smallskip
\item[(a)] for $i,j\in\{ 1,2,\dots,k^{l-2-m}\}$ with $i\neq j$, 
$$
f\big(V^i_{(l+1)s+r+1}\big) \cap f\big(V^j_{(l+1)s+r+1}\big) = V^i_{(l+1)s+r} \cap V^j_{(l+1)s+r}=\emptyset
$$
(by the induction hypothesis), and so 
$$
V^i_{(l+1)s+r+1} \cap V^j_{(l+1)s+r+1} = \emptyset,
$$

\smallskip
\item[(b)] $V_{(l+1)s+r+1} \subseteq W^m(p)$,

\smallskip
\item[(c)] if $r=l$, then $V_{(l+1)s+r+1} \subseteq W^l(p)$.
\end{enumerate}
Thus
$$
\card V_{(l+1)s+r+1} = \card V_{(l+1)s+r}.
$$

We only need to verify the last property required for $V_{(l+1)s+r+1}$ in the induction step now. Indeed, suppose that $x,y\in V_{(l+1)s+r+1}$ with $x\neq y$ and $\{x,y\}\subseteq A$ for some $A\in\bigvee\limits_{i=0}^{(l+1)s+r} f^{-i}\(\W^l\)$. Then by construction $f(x),f(y)$, and $f(A)$ satisfy
\begin{enumerate}
\smallskip
\item[(a)] $f(A)\subseteq B$ for some $B\in \bigvee\limits_{i=0}^{(l+1)s+r-1} f^{-i}\(\W^l\)$,

\smallskip
\item[(b)] $f(x), f(y)\in V_{(l+1)s+r}$, and $f(x)\neq f(y)$,

\smallskip
\item[(c)] $\{f(x),f(y)\} \subseteq f(A)\subseteq B$.
\end{enumerate}
This contradicts property~(4) for $V_{(l+1)s+r}$ in the induction hypothesis.

\smallskip

The recursive construction and the inductive proof of the properties of the construction are now complete.

Note that by our construction, we have
\begin{equation}    \label{eqCardUn}
\card V_{(l+1)s} = k^{(l-m-2)s}, \qquad s\in \N.
\end{equation}

\smallskip

For each $s\in\N$, we consider
$$
B_{(l+1)s}= \bigcap\limits_{j=0}^{(l+1)s-1} f^{-j} \(W^m\(v_{(l+1)s - j}\) \)  \in \bigvee\limits_{j=0}^{(l+1)s-1} f^{-j} \(\W^m \).
$$
Then $V_{(l+1)s} \subseteq B_{(l+1)s}$ by properties~(1) and (2) of the construction. On the other hand, by property~(4), if $\mathcal{A}\subseteq \bigvee\limits_{j=0}^{(l+1)s-1} f^{-j} \(\W^l \)$ satisfies
$$
\bigcup \mathcal{A} \supseteq B_{(l+1)s} \supseteq V_{(l+1)s}.
$$
So $\card \mathcal{A} \geq \card V_{(l+1)s}$.

Thus by (\ref{eqDefTopTailEntropy}), (\ref{eqH(xi|eta)}), and (\ref{eqCardUn}),
\begin{align*}
h^*(f)  & = \lim\limits_{m\to+\infty}  \lim\limits_{l\to+\infty} \lim\limits_{n\to+\infty}  \frac{1}{n} H \( \bigvee\limits_{i=0}^{n-1} f^{-i}\(\W^l\) \Bigg| \bigvee\limits_{j=0}^{n-1} f^{-j}\(\W^m\)  \)   \\
        & \geq \liminf\limits_{m\to+\infty}  \liminf\limits_{l\to+\infty} \liminf\limits_{s\to+\infty}  \frac{1}{(l+1)s} \log \(k^{(l-m-2)s}\)  \\
        & = \liminf\limits_{m\to+\infty}  \liminf\limits_{l\to+\infty}  \frac{l-m-2}{l+1} \log k \\
        & = \log k  \\
        & > 0.
\end{align*}
Therefore, the map $f$ is not asymptotically $h$-expansive.
\end{proof}

\begin{lemma}   \label{lmHExpIterates}
Let $g\: X\rightarrow X$ be a continuous map on a compact metric space $(X,d)$. If $g$ is $h$-expansive then so is $g^n$ for each $n\in\N$. 
\end{lemma}

The converse can also be easily established, i.e., if $g^n$ is $h$-expansive for some $n\in\N$, then so is $g$. But we will not need it in this paper.

\begin{proof}
We first observe from Definition~\ref{defRefSeqOpenCover} that if $\{\xi_l\}_{l\in\N_0}$ is a refining sequence of open covers, then so is $\{\xi^n_l\}_{l\in\N_0}$ for each $n\in\N$, where $\xi^n_l = \bigvee\limits_{i=0}^{n-1} g^{-i}\(\xi_l\)$. We also note that given an open cover $\lambda$ of $X$, we have
$$
\bigvee\limits_{i=0}^{mn-1} g^{-i}\(\lambda\) =\bigvee\limits_{j=0}^{m-1} \(g^n\)^{-j}\(\lambda^n\)
$$ 
for $n,m\in \N$, where $\lambda^n = \bigvee\limits_{k=0}^{n-1} g^{-k}\(\lambda\)$. 

Assume that $g$ is $h$-expansive, then $h(g|\lambda)=0$ for some finite open cover $\lambda$ of $X$. Thus for each $n\in\N$,
\begin{align*}
 h(g|\lambda)  =  & \lim\limits_{l\to+\infty} \lim\limits_{m\to+\infty}  \frac{1}{mn} H \( \bigvee\limits_{i=0}^{mn-1} g^{-i}\(\xi_l\) \Bigg| \bigvee\limits_{j=0}^{mn-1} g^{-j}\(\lambda\)  \)  \\
=  & \frac1n \lim\limits_{l\to+\infty} \lim\limits_{m\to+\infty}  \frac{1}{m} H \( \bigvee\limits_{i=0}^{m-1} \(g^n\)^{-i}\(\xi^n_l\) \Bigg| \bigvee\limits_{j=0}^{m-1} \(g^n\)^{-j}\(\lambda^n\)  \)  \\
=  & \frac1n h\(g^n|\lambda^n\), 
\end{align*}
where $\xi^n_l$, $\lambda^n$ are defined as above. Note that $\lambda^n$ is also a finite open cover of $X$. Therefore $h\(g^n|\lambda^n\) = 0$, i.e., $g^n$ is $h$-expansive.
\end{proof}

The proof of the following theorem is similar to that of Theorem~\ref{thmNotAsympHExp}, and slightly simpler. However, due to subtle differences in both notation and constructions, we include the proof for the convenience of the reader.

\begin{theorem}  \label{thmNotHExp}
No expanding Thurston map is $h$-expansive.
\end{theorem}

\begin{proof}
Let $f$ be an expanding Thurston map.

By Theorem~\ref{thmNotAsympHExp} and the fact that if $f$ is $h$-expanding then it is asymptotically $h$-expansive (see \cite[Corollary~2.1]{Mi76}), we can assume that $f$ has no periodic critical points.

Note that by (\ref{eqLocalDegreeProduct}), if a point $x\in S^2$ is a periodic critical point of $f^i$ for some $i\in\N$, then there exists $j\in\N_0$ such that $f^j(x)$ is a periodic critical point of $f$. So $f^i$ has no periodic critical points for $i\in\N$.

By Lemma~\ref{lmHExpIterates}, it suffices to prove that there exists $i\in\N$ such that $f^i$ is not $h$-expansive. Thus by Lemma~\ref{lmCexistsL}, we can assume, without loss of generality, that there exists a Jordan curve $\CC\subseteq S^2$ containing $\post f$ such that $f(\CC)\subseteq \CC$ and no $1$-tile joins opposite sides of $\CC$.

In addition, we can assume, without loss of generality, that there exists a critical point $p\in\crit f \setminus \CC$ with $f^2(p)=f(p) \neq p$. Indeed, we can choose any critical point $p_0\in \crit f$, then $f^{2i}(p_0)=f^i(p_0)\neq p_0$ for some $i\in\N$ since $f$ has no periodic critical points. By Lemma~\ref{lmPreImageDense}, there exist $j\in\N$ and $p\in f^{-ij}(p_0)\setminus \CC$. We replace $f$ by $f^{i(j+1)}$. Note that for this new map $f$, we have $p\in\crit f \setminus \CC$, $f^2(p)=f(p)\neq p$, $f(\CC)\subseteq \CC$ and no $1$-tile joins opposite sides of $\CC$.

Let $k=\deg_f(p)$. Then $k>1$.

From now on, we consider the cell decompositions of $S^2$ induced by $f$ and $\CC$ in this proof.

\smallskip

Recall that $\W^i$ defined in (\ref{defSetNFlower}) denotes the set of all $i$-flowers $W^i(v)$, $v\in\V^i$, for each $i\in\N_0$.

Since $f$ is expanding, it is easy to see from Lemma~\ref{lmCellBoundsBM}, Proposition~\ref{propCellDecomp}, and the Lebesgue Number Lemma (\cite[Lemma~27.5]{Mu00}) that $\{ \W^i \}_{i\in\N_0}$ forms a refining sequence of open covers of $S^2$ (see Definition~\ref{defRefSeqOpenCover}). Thus it suffices to prove that 
\begin{equation*}
h(f|\W^m)=\lim\limits_{l\to+\infty} \lim\limits_{n\to+\infty}  \frac{1}{n} H \( \bigvee\limits_{i=0}^{n-1} f^{-i}\(\W^l\) \Bigg| \bigvee\limits_{j=0}^{n-1} f^{-j}\(\W^m\)  \) > 0
\end{equation*}
for each $m\in\N$ sufficient large. See (\ref{eqH(xi|eta)}) for the definition of $H$.

Our plan is to construct a sequence $\{v_i\}_{i\in\N_0}$ of $m$-vertices such that for each $n\in\N_0$, the number of elements in $\bigvee\limits_{i=0}^{n-1} f^{-i}\(\W^l\)$ needed to cover $B_n=\bigcap\limits_{j=0}^{n-1} f^{-j} (W^m(v_{n-j}))$ can be bounded from below in such a way that $h(f|\W^m)>0$ follows immediately. More precisely, we observe that the more connected components $B_n$ has, the harder to cover $B_n$. So we will choose $\{v_i\}_{i\in\N_0}$ as a periodic sequence of $m$-vertices shadowing an infinite backward pseudo-orbit under iterations of $f$ in such a way that each period of $\{v_i\}_{i\in\N_0}$ begins with a backward orbit starting at $f(p)$ and $p$, and approaching $f(p)$ as the index $i$ increases, and then ends with $f(p)$. By a recursive construction, we keep track of each $B_n$ by a finite subset $V_n\subseteq B_n$ with the property that $\card(A\cap V_n)\leq 1$ for each $A \in \bigvee\limits_{i=0}^{n-1} f^{-i}\(\W^l\)$. A quantitative control of the size of $V_n$ leads to the conclusion that $h(f|\W^m)>0$ for some $m$ sufficiently large.

\smallskip

For this we fix $m,l\in\N$ with $l>2m+100 >200$.

Define $q_1 = p$. Then $q_1$ is necessarily a $1$-vertex, but not a $0$-vertex, i.e., $q_1\in \V^1 \setminus \V^0$. Since $q_1=p\notin \CC$, we have $q_1\in W^0(f(p))$. By (\ref{defFlower}), the only $2$-vertex contained in $W^2(f(p))$ is $f(p)$. So $q_1 \in W^0(f(p)) \setminus W^2(f(p))$. Since $f(W^i(f(p))) = W^{i-1}(f(p))$ for each $i\in \N$ (see Remark~\ref{rmFlower}), we can recursively choose $q_j\in\V^j$ for each $j\in\{2,3,\dots,m+2\}$ such that
\begin{enumerate}

\smallskip
\item[(i)] $f(q_j) = q_{j-1}$,

\smallskip
\item[(ii)] $q_j \in W^{j-1}(f(p)) \setminus W^{j+1}(f(p))$.
\end{enumerate}
Set $q_0 = q_{m+2}$.

Since $f(W^i(p)) = W^{i-1}(f(p))$ for each $i\in \N$, and $k=\deg_f(p)>1$, we can choose distinct points $p_i\in\V^{m+3}$, $i\in\{1,2,\dots,k\}$, such that
\begin{enumerate}

\smallskip
\item[(i)] $f(p_i)= q_{m+2}$,

\smallskip
\item[(ii)] $p_i \in W^{m+2}(p) \setminus W^{m+4}(p)$.
\end{enumerate}

We will now construct recursively, for each $n=(m+2)s+r$ with $s\in\N_0$ and $r\in\{0,1,\dots,m+1\}$, an $m$-vertex $v_n\in\V^m$ and a set of $n$-vertices $V_n\subseteq \V^n$ such that for each $n\in\N_0$, the following properties are satisfied:
\begin{enumerate}

\smallskip
\item[(1)] $V_n \subseteq W^m(v_n)$;

\smallskip
\item[(2)] $f(V_n) = V_{n-1}$ if $n\neq 0$;

\smallskip
\item[(3)]
\begin{enumerate}
\item[(i)] $V_n \subseteq W^{m+1+r}(q_r)$ if $n=(m+2)s+r$ for some $s\in\N_0$ and some $r\in\{1,2,\dots,m+1\}$, 

\smallskip
\item[(ii)] $V_n\subseteq W^{m+1+m+2}(q_0)$ if $n=(m+2)s$ for some $s\in\N_0$;
\end{enumerate}

\smallskip
\item[(4)] $\card V_n = k^{ \lceil \frac{n}{m+2} \rceil}$;

\smallskip
\item[(5)] for $A\in \bigvee\limits_{i=0}^{n-1} f^{-i} \(\W^l\)$ and $x,y\in V_n$ with $x\neq y$, we have $\{x,y\}\nsubseteq A$.
\end{enumerate}

We start our construction by first defining $v_n\in\V^m$ for each $n\in\N_0$. For $s\in\N_0$ and $r\in\{1,2,\dots,m\}$, set $v_{(m+2)s+r} = q_r$. For $s\in\N_0$ and $r\in\{0, m+1\}$, set $v_{(m+2)s+r} = f(p)$.

We now define $V_n$ recursively.

Let $V_0 = \{q_{m+2}\}$. Clearly $V_0$ satisfies properties~(1) through (5) in the induction step.

Assume that $V_n$ is defined and satisfies properties~(1) through (5) for each $n\in\{0,1,\dots,(m+2)s+r\}$, where $s\in\N_0$ and $r\in\{0,1,\dots,m+1\}$, we continue our construction in the following cases depending on $r$.

\smallskip

\emph{Case 1.} Assume $r=0$. Then $v_{(m+2)s+r} = f(p)$ and $v_{(m+2)s+r+1} = q_1=p$.

Note that $V_{(m+2)s+r} \subseteq W^{2m+3}(q_r)$ by the induction hypothesis, $q_r = q_{m+2} \in W^{m+1}(f(p))$, $f(p_i) = q_r$, and $f\(W^{2m+4}(p_i)\) = W^{2m+3}(q_r)$ for each $i\in\{1,2,\dots,k\}$ (see Remark~\ref{rmFlower}). Fix an arbitrary $i\in\{1,2,\dots,k\}$. We can choose, for each $x\in V_{(m+2)s+r}$, a point $x'\in W^{2m+4}(p_i)$ such that $f(x')=x$. Then define $V^i_{(m+2)s+r+1}$ to be the collection of all such chosen $x'$ that corresponds to $x\in V_{(m+2)s+r}$. Set 
$$
V_{(m+2)s+r+1} = \bigcup\limits_{i=1}^k V^i_{(m+2)s+r+1}.
$$

Since $p_i \in W^{m+2}(p)$ and $V^i_{(m+2)s+r+1} \subseteq W^{2m+4}(p_i)$, we get that $V^i_{(m+2)s+r+1} \subseteq W^{m+2}(p)$. So $V_{(m+2)s+r+1} \subseteq W^{m+2}(p) \subseteq W^m(p)$. Since $v_{(m+2)s+r+1} = q_1=p$, properties~(1) and (3) are verified. Property~(2) is clear from the construction. 

To establish property~(4), it suffices to show that $V^i_{(m+2)s+r+1} \cap V^j_{(m+2)s+r+1} = \emptyset$ for $1\leq i<j \leq k$. Indeed, since $V^i_{(m+2)s+r+1} \subseteq W^{2m+4}(p_i)$ and $V^j_{(m+2)s+r+1} \subseteq W^{2m+4}(p_j)$, it suffices to prove that $\overline{W}^{2m+4}(p_i)\cap \overline{W}^{2m+4}(p_j) = \emptyset$. Suppose that $\overline{W}^{2m+4}(p_i)\cap \overline{W}^{2m+4}(p_j) \neq \emptyset$, then since no $1$-tile joins opposite sides of $\CC$, and $p_i,p_j\in\V^{m+3}$, we get from Lemma~\ref{lmTileInFlower}(iii) that $p_i=p_j$, i.e., $i=j$. But $i<j$, a contradiction.

We only need to verify property~(5) now. Indeed, suppose that distinct points $x,y\in V_{(m+2)s+r+1}$ satisfy $\{x,y\}\subseteq A$ for some $A\in\bigvee\limits_{a=0}^{(m+2)s+r} f^{-a}\(\W^l\)$. Then $A\subseteq W^l(v^l)$ for some $v^l\in\V^l$. By construction, there exist $i,j\in\{1,2,\dots,k\}$ such that $x\in W^{2m+4}(p_i)$ and $y\in W^{2m+4}(p_j)$. Since $l> 2m+100$ and $x\in W^l(v^l)\cap W^{2m+4}(p_i)$, we get $v^l \in \overline{W}^{2m+4}(p_i)$ by Lemma~\ref{lmTileInFlower}(i). Similarly $v^l\in \overline{W}^{2m+4}(p_j)$. Then by the argument above, we get that $p_i=p_j$, i.e., $i=j$. Thus $f(x)\neq f(y)$ by construction. But then $f(x),f(y)$, and $f(A)$ satisfy
\begin{enumerate}

\smallskip
\item[(a)] $f(A)\subseteq B$ for some $B\in \bigvee\limits_{a=0}^{(m+2)s+r-1} f^{-a}\(\W^l\)$,

\smallskip
\item[(b)] $f(x), f(y)\in V_{(m+2)s+r}$, and $f(x)\neq f(y)$,

\smallskip
\item[(c)] $\{f(x),f(y)\} \subseteq f(A)\subseteq B$.
\end{enumerate}
This contradicts property~(5) for $V_{(m+2)s+r}$ in the induction hypothesis.

\smallskip

\emph{Case 2.} Assume $r\neq 0$, i.e., $r\in\{1,2,\dots,m+1\}$.

Note that $V_{(m+2)s+r} \subseteq W^{m+1+r} (q_r)$, $f(q_{r+1})=q_r$, and by Remark~\ref{rmFlower}, $
f\(W^{m+1+r+1}(q_{r+1})\) = W^{m+1+r}(q_r)$. We can choose, for each $x\in V_{(m+2)s+r}$, a point $x'\in W^{m+1+r+1}(q_{r+1})$ such that $f(x')=x$. Then define $V_{(m+2)s+r+1}$ to be the collection of all such chosen $x'$ that corresponds to $x\in V_{(m+2)s+r}$. Properties~(2), (3), and (4) are clear from the construction. To establish property~(1) in the case when $r\in\{1,2,\dots,m-1\}$, we recall that $v_{(m+2)s+r+1} = q_{r+1}$. For the case when $r\in\{m,m+1\}$, we note that $V_{(m+2)s+r+1} \subseteq W^{m+1+r+1}(q_{r+1})$ and $q_{r+1}\in W^r(f(p))$, so
$$
V_{(m+2)s+r+1} \subseteq W^{2m}(q_{r+1}) \subseteq W^m(f(p)) = W^m\(v_{(m+2)s+r+1}\).
$$

We only need to verify property~(5) now. Indeed, suppose that distinct points $x,y\in V_{(m+2)s+r+1}$ satisfy $\{x,y\}\subseteq A$ for some $A\in\bigvee\limits_{i=0}^{(m+2)s+r} f^{-i}\(\W^l\)$. Then by construction $f(x),f(y)$, and $f(A)$ satisfy
\begin{enumerate}

\smallskip
\item[(a)] $f(A)\subseteq B$ for some $B\in \bigvee\limits_{i=0}^{(m+2)s+r-1} f^{-i}\(\W^l\)$,

\smallskip
\item[(b)] $f(x), f(y)\in V_{(m+2)s+r}$, and $f(x)\neq f(y)$,

\smallskip
\item[(c)] $\{f(x),f(y)\} \subseteq f(A)\subseteq B$.
\end{enumerate}
This contradicts property~(5) for $V_{(m+2)s+r}$ in the induction hypothesis.

\smallskip

The recursive construction and the inductive proof of the properties of the construction are now complete.

\smallskip

For each $s\in\N$, we consider
$$
B_{(m+2)s}= \bigcap\limits_{j=0}^{(m+2)s-1} f^{-j} \(W^m\(v_{(m+2)s - j}\) \)  \in \bigvee\limits_{j=0}^{(m+2)s-1} f^{-j} \(\W^m \).
$$
Then $V_{(m+2)s} \subseteq B_{(m+2)s}$ by properties~(1) and (2) of the construction. On the other hand, by property~(5), if $\mathcal{A}\subseteq \bigvee\limits_{j=0}^{(m+2)s-1} f^{-j} \(\W^l \)$ satisfies
$$
\bigcup \mathcal{A} \supseteq B_{(m+2)s} \supseteq V_{(m+2)s}.
$$
So $
\card \mathcal{A} \geq \card V_{(m+2)s} = k^s$, where the equality follows from property~(4).

Thus by (\ref{eqH(xi|eta)}),
\begin{align*}
h(f|\W^m)= & \lim\limits_{l\to+\infty} \lim\limits_{n\to+\infty}  \frac{1}{n} H \( \bigvee\limits_{i=0}^{n-1} f^{-i}\(\W^l\) \Bigg| \bigvee\limits_{j=0}^{n-1} f^{-j}\(\W^m\)  \)   \\
\geq  & \liminf\limits_{l\to+\infty} \liminf\limits_{s\to+\infty}  \frac{1}{(m+2)s} \log \(k^s\)  
 =    \frac{\log k}{m+2}  
 >     0.
\end{align*}

Therefore, the map $f$ is not $h$-expansive.
\end{proof}

\section{Large deviation principles}    \label{sctLDP}

\subsection{Thermodynamical formalism}   \label{subsctThermDynForm}

We review some key concepts from thermodynamical formalism. For a more careful introduction in our context, see \cite[Section~5]{Li14}. We refer the reader to \cite[Chapter~3]{PU10}, \cite[Chapter~9]{Wa82} or \cite[Chapter~20]{KH95} for a detailed study of these concepts for general dynamical systems.

Let $(X,d)$ be a compact metric space and $g\:X\rightarrow X$ a continuous map. For $n\in\N$ and $x,y\in X$,
$$
d^n_g(x,y)=\operatorname{max}\big\{ \hspace{-0.2mm} d\!\(g^k(x),g^k(y)\) \hspace{-0.3mm} \big| k\in\{0,1,\dots,n-1\} \!\big\}
$$
defines a new metric on $X$. A set $F\subseteq X$ is \defn{$(n,\epsilon)$-separated}, for some $n\in\N$ and $\epsilon>0$, if for each pair of distinct points $x,y\in F$, we have $d^n_g(x,y)\geq \epsilon$. For $\epsilon > 0$ and $n\in\N$, let $F_n(\epsilon)$ be a maximal (in the sense of inclusion) $(n,\epsilon)$-separated set in $X$.

For each $\psi \in\CCC(X)$, the following limits exist and are equal, and we denote the limits by $P(g,\psi)$ (see for example, \cite[Theorem~3.3.2]{PU10}):
\begin{align}  \label{defTopPressure}
P(g,\psi)  & = \lim \limits_{\epsilon\to 0} \limsup\limits_{n\to+\infty} \frac{1}{n} \log  \sum\limits_{x\in F_n(\epsilon)} \exp(S_n \psi(x)) \notag \\
           & = \lim \limits_{\epsilon\to 0} \liminf\limits_{n\to+\infty} \frac{1}{n} \log  \sum\limits_{x\in F_n(\epsilon)} \exp(S_n \psi(x)), 
\end{align}
where $S_n \psi (x) = \sum\limits_{j=0}^{n-1} \psi(g^j(x))$ is defined in (\ref{eqDefSnPt}). We call $P(g,\psi)$ the \defn{topological pressure} of $g$ with respect to the \emph{potential} $\psi$. When $\psi\equiv 0$, we call $P(g,0)=h_{\operatorname{top}}(g)$ the \defn{topological entropy} of $g$.

Let $\mu\in \MMM(X,g)$. We denote by $h_\mu(g)$ the \emph{measure-theoretic entropy} of $g$ for $\mu$. Then for each $\psi\in\CCC(X)$, the \defn{measure-theoretic pressure} $P_\mu(g,\psi)$ of $g$ for the measure $\mu$ and the potential $\psi$ is
\begin{equation}  \label{defMeasTheoPressure}
P_\mu(g,\psi)= h_\mu (g) + \int \! \psi \,\mathrm{d}\mu.
\end{equation}

By the Variational Principle (see for example, \cite[Theorem~3.4.1]{PU10}), we have that for each $\psi\in\CCC(X)$,
\begin{equation}  \label{eqVPPressure}
P(g,\psi)=\sup\{P_\mu(g,\psi)\,|\,\mu\in \MMM(X,g)\},
\end{equation}
and in particular,
\begin{equation}  \label{eqVPEntropy}
h_{\operatorname{top}}(g)=\sup\{h_\mu(g)\,|\,\mu\in \MMM(X,g)\}.
\end{equation}
A measure $\mu$ that attains the supremum in (\ref{eqVPPressure}) is called an \defn{equilibrium state} for the transformation $g$ and the potential $\psi$. A measure $\mu$ that attains the supremum in (\ref{eqVPEntropy}) is called a \defn{measure of maximal entropy} of $g$.

By the work of P.~Ha\"{\i}ssinsky and K.~Pilgrim \cite{HP09}, and M.~Bonk and D.~Meyer \cite{BM10}, we know that there exists a unique measure of maximal entropy $\mu_f$ for an expanding Thurston map $f$, and that
\begin{equation}   \label{eqTopEntropy}
h_{\operatorname{top}} (f)=\log(\deg f).
\end{equation}

The existence and uniqueness of the equilibrium state for an expanding Thurston map $f$ and a H\"older continuous potential $\phi$ is established in \cite[Theorem~1.1]{Li14}.

\begin{theorem}   \label{thmExistenceUniquenessES}
Let $f\:S^2 \rightarrow S^2$ be an expanding Thurston map and $d$ be a visual metric on $S^2$ for $f$. Let $\phi$ be a real-valued H\"{o}lder continuous function on $S^2$ with respect to the metric $d$. Then there exists a unique equilibrium state $\mu_\phi$ for the map $f$ and the potential $\phi$. 
\end{theorem}

Let $f$, $d$, $\phi$, $\alpha$ satisfy the Assumptions. Recall that the main tool we used in \cite{Li14} to prove this theorem is the Ruelle operator $\RR_\phi$. We summarize relevant definitions and facts about the Ruelle operator below and refer the reader to \cite[Section~5]{Li14} for a detailed discussion.

Let $\psi\in \CCC(S^2)$ be a continuous function. Recall that $f\: S^2 \rightarrow S^2$ is an expanding Thurston map. The \defn{Ruelle operator} $\RR_\psi$ on $\CCC(S^2)$ is defined as the following
\begin{equation}   \label{eqDefRuelleOp}
\RR_\psi(u)(x)= \sum\limits_{y\in f^{-1}(x)}  \deg_f(y) u(y) \exp(\psi(y)),
\end{equation}
for each $u\in \CCC(S^2)$. Note that $\RR_\psi$ is a well-defined, positive, continuous operator on $\CCC(S^2)$. The adjoint operator $\RR_\psi^*\: \CCC^*(S^2)\rightarrow \CCC^*(S^2)$ of $\RR_\psi$ acts on the dual space $\CCC^*(S^2)$ of the Banach space $\CCC(S^2)$. We identify $\CCC^*(S^2)$ with the space $\MMM(S^2)$ of finite signed Borel measures on $S^2$ by the Riesz representation theorem.

\begin{proof}[Proof of Theorem~\ref{thmExistenceES}]
By Alaoglu's theorem, the space $\MMM(S^2,f)$ of $f$-invariant Borel probability measures equipped with the weak$^*$ topology is compact. Since the measure-theoretic entropy $\mu\mapsto h_\mu(f)$ is upper semi-continuous by Corollary~\ref{corUSC}, so is $\mu\mapsto P_\mu(f,\psi)$ by (\ref{defMeasTheoPressure}). Thus $\mu\mapsto P_\mu(f,\psi)$ attains its supremum over $\MMM(S^2,f)$ at a measure $\mu_\psi$, which by the Variational Principle (\ref{eqVPPressure}) is an equilibrium state for the map $f$ and the potential $\psi$.
\end{proof}

\subsection{Level-2 large deviation principles}   \label{subsctLDP}

Let $X$ be a compact metrizable topological space. Recall that $\PPP(X)$ is the set of Borel probability measures on $X$. We equip $\PPP(X)$ with the weak$^*$ topology.  Note this topology is metrizable (see for example, \cite[Theorem~5.1]{Con85}). Let $I\: \PPP(X) \rightarrow [0,+\infty]$ be a \defn{lower semi-continuous function}, i.e., $I$ satisfy the condition that $\liminf_{y\to x} I(y)\geq I(x)$ for all $x\in \PPP(X)$.

A sequence $\{\Omega_n \}_{n\in\N}$ of Borel probability measures on $\PPP(X)$ is said to satisfy a \defn{large deviation principle with rate function $I$} if for each closed subset $\FF$ of $\PPP(X)$ and each open subset $\GG$ of $\PPP(X)$ we have 
$$
\limsup\limits_{n\to+\infty} \frac1n \log \Omega_n(\FF) \leq - \inf\{I(x) \,|\, x\in\FF \},
$$
and
$$
\liminf\limits_{n\to+\infty} \frac1n \log \Omega_n(\GG) \geq - \inf\{I(x) \,|\, x\in\GG \}.
$$

We will apply the following theorem due to Y.~Kifer \cite[Theorem~4.3]{Ki90}, reformulated by H.~Comman and J.~Rivera-Letelier \cite[Theorem~C]{CRL11}.

\begin{theorem}[Y.~Kifer, 1990; H.~Comman \& J.~Rivera-Letelier, 2011]    \label{thmAbsLargeDeviationPrincipleCRL}
Let $X$ be a compact metrizable topological space, and let $g\: X\rightarrow X$ be a continuous map. Fix $\phi\in\CCC(X)$, and let $H$ be a dense vector subspace of $\CCC(X)$ with respect to the uniform norm. Let $I^\phi \: \PPP(X) \rightarrow [0,+\infty]$ be the function defined by
\begin{equation*}
I^\phi (\mu) = 
\left\{
	\begin{array}{ll}
		P(g,\phi) -\int\! \phi \,\mathrm{d}\mu - h_\mu(g)  & \mbox{if } \mu\in\MMM(X,g); \\
		+\infty & \mbox{if } \mu\in\PPP(X) \setminus \MMM(X,g).
	\end{array}
\right.
\end{equation*}
We assume the following conditions are satisfied:
\begin{enumerate}

\smallskip
\item[(i)] The measure-theoretic entropy $h_\mu(g)$ of $g$, as a function of $\mu$ defined on $\MMM(X,g)$ (equipped with the weak$^*$ topology), is finite and upper semi-continuous.

\smallskip
\item[(ii)] For each $\psi \in H$, there exists a unique equilibrium state for the map $g$ and the potential $\phi+\psi$.
\end{enumerate}

Then every sequence $\{ \Omega_n \}_{n\in\N}$ of Borel probability measures on $\PPP(X)$ such that for each $\psi \in H$,
\begin{equation}   \label{eqThmCCondition}
\lim\limits_{n\to+\infty} \frac{1}{n} \log \int_{\PPP(X)} \! \exp \(n\int\!\psi \,\mathrm{d}\mu \) \,\mathrm{d}\Omega_n(\mu) = P(g,\phi+\psi) - P(g,\phi),
\end{equation}
satisfies a large deviation principle with rate function $I^\phi$, and it converges in the weak$^*$ topology to the Dirac measure supported on the unique equilibrium state for the map $g$ and the potential $\phi$. Furthermore, for each convex open subset $\GG$ of $\PPP(X)$ containing some invariant measure, we have
\begin{equation*}
\lim\limits_{n\to+\infty}  \frac{1}{n} \log \Omega_n(\GG) 
= \lim\limits_{n\to+\infty}  \frac{1}{n} \log \Omega_n\big(\overline{\GG}\big) 
= -\inf\limits_{\GG} I^\phi 
= -\inf\limits_{\overline{\GG}} I^\phi.
\end{equation*}
\end{theorem}

Recall that $P(g,\phi)$ is the topological pressure of the map $g$ with respect to the potential $\phi$.

In our context, $X=S^2$, the map $g=f$ where $f\: S^2\rightarrow S^2$ is an expanding Thurston map with no periodic critical points. Fix a visual metric $d$ on $S^2$ for $f$. The function $\phi$ is a real-valued H\"older continuous function with an exponent $\alpha\in (0,1]$. Then $H=\Holder{\alpha}(S^2,d)$ is the space of real-valued H\"older continuous functions with the exponent $\alpha$ on $(S^2,d)$. Note that $\Holder{\alpha}(S^2,d)$ is dense in $\CCC(S^2)$ (equipped with the uniform norm) (see for example, \cite[Lemma~6.12]{Li14}). Condition~(i) is satisfied by Corollary~\ref{corUSC}. Condition~(ii) is guaranteed by Theorem~\ref{thmExistenceUniquenessES}. Thus we just need to verify (\ref{eqThmCCondition}) for the sequences we will consider in this section.

\subsection{Technical lemmas}    \label{subsctLDPLemmas}
\begin{lemma}    \label{lmAtMostOneCritPerTile}
Let $f$, $\CC$, $d$, $\Lambda$ satisfy the Assumptions. Then there exists $N_0\in\N$ such that for each $n\geq N_0$ and each $n$-tile $X^n\in\X^n(f,\CC)$, the number of fixed points of $f^n$ contained in $X^n$ is at most $1$.
\end{lemma}

\begin{proof}
By Lemma~\ref{lmMetricDistortion}, for each $i\in\N$, each $i$-tile $X^i\in\X^i(f,\CC)$, and each pair of points $x,y\in X^i$, we have
$$
d\(f^i(x),f^i(y)\) \geq \frac{\Lambda^i}{C_0} d(x,y),
$$
where $C_0>1$ is a constant depending only on $f$ and $d$ from Lemma~\ref{lmMetricDistortion}.

We choose $N_0\in\N$ such that $\Lambda^{N_0} > C_0$.

Let $n\geq N_0$ and $X^n\in\X^n(f,\CC)$. Suppose two distinct points $p,q\in X$ satisfy $f^n(p)=p$ and $f^n(q)=q$. Then
$$
1=\frac{d\(f^n(p),f^n(q)\)}{d(p,q)} \geq \frac{\Lambda^n}{C_0} >1,
$$
a contradiction. This completes the proof.
\end{proof}

This lemma in some sense generalizes Lemma~4.3 in \cite{Li13} to Jordan curves that are not necessarily $f$-invariant, but $f^{n_c}$-invariant for some $n_c\in\N$. The conclusions of both lemmas hold when $n$ is sufficiently large, which is a combinatorial condition in Lemma~4.3 in \cite{Li13} and a metric condition for Lemma~\ref{lmAtMostOneCritPerTile} here. The proof of Lemma~\ref{lmAtMostOneCritPerTile} is simpler, but the proof of Lemma~4.3 in \cite{Li13} is more self-contained.

By Lemma~5.1 and Lemma~5.2 in \cite{Li14}, we have the following distortion bounds.

\begin{lemma}   \label{lmSnPhiBound}
Let $f$, $\CC$, $d$, $L$, $\Lambda$, $\phi$, $\alpha$ satisfy the Assumptions. Then there exist constants $C_1>0$ and $C_2\geq 1$ depending only on $f$, $\CC$, $d$, $\phi$, and $\alpha$ such that
\begin{equation}  \label{eqSnPhiBound}
\Abs{S_m\phi(x)-S_m\phi(y)}  \leq C_1 d(f^m(x),f^m(y))^\alpha,
\end{equation}
and
\begin{equation}   \label{eqSigmaExpSnPhiBound}
\frac{\sum\limits_{x'\in f^{-l}(x_0)}  \deg_{f^l}(x')  \exp (S_l\phi(x'))}{\sum\limits_{y'\in f^{-l}(y_0)}  \deg_{f^l}(y')  \exp (S_l\phi(y'))}\leq C_2,
\end{equation}
for $n,m,l\in\N_0$ with $m\leq n $, $X^n\in\X^n(f,\CC)$, $x,y\in X^n$, and $x_0,y_0\in S^2$. 
\end{lemma}

Note that due to the convention described in Section~\ref{sctAssumptions}, we do not say that $C_1$ and $C_2$ depend on $\Lambda$ or $n_\CC$.

We record the following well-known lemma, sometimes known as the Portmanteau Theorem, and refer the reader to \cite[Theorem~2.1]{Bi99} for a proof.

\begin{lemma}   \label{lmPortmanteau}
Let $(X,d)$ be a compact metric space, and $\mu$ and $\mu_i$, for $i\in\N$, be Borel probability measures on $X$. Then the following are equivalent:
\begin{enumerate}
\smallskip

\item[(i)] $\mu_i \stackrel{w^*}{\longrightarrow} \mu$ as $i\longrightarrow +\infty$;

\smallskip

\item[(ii)] $\limsup\limits_{i\to+\infty} \mu_i(F) \leq \mu(F)$ for each closed set $F\subseteq X$;

\smallskip

\item[(iii)] $\liminf\limits_{i\to+\infty} \mu_i(G) \geq \mu(G)$ for each open set $G\subseteq X$;

\smallskip

\item[(iv)] $\lim\limits_{i\to+\infty} \mu_i(B)= \mu(B)$ for each Borel set $B\subseteq X$ with $\mu(\partial B) = 0$.
\end{enumerate}

\end{lemma}

\subsection{Characterizations of the pressure $P(f,\phi)$}    \label{subsctTopPres}
Let $f$, $d$, $\phi$, $\alpha$ satisfy the Assumptions. We denote by $m_\phi$ the unique eigenmeasure of $\RR^*_\phi$, i.e., the unique Borel probability measure on $S^2$ that satisfies $\RR^*_\phi(m_\phi)=c m_\phi$ for some constant $c\in \R$ (compare \cite[Theorem~5.11 and Corollary~6.10]{Li14}). As in \cite{Li14}, we denote $\overline\phi=\phi-P(f,\phi)$, and
\begin{equation}  \label{eqDefPhiTilde}
\widetilde\phi = \phi - P(f,\phi) +\log u_\phi - \log (u_\phi\circ f),
\end{equation}
where $u_\phi$ is the unique eigenfunction, upto scalar multiplication, of $\RR_{\overline\phi}$ corresponding to the eigenvalue $1$, which is also the uniform limit of the sequence of continuous functions $\Big\{ \frac{1}{n} \sum\limits_{j=0}^{n-1} \RR^j_{\overline\phi} (\mathbbm{1}_{S^2})  \Big\}_{n\in\N}$ (compare Theorem~5.16 and the remark that follows it in \cite{Li14}).

We record the following result on equidistribution of preimages with respect to the equilibrium state $\mu_\phi$ and the measure $m_\phi$ (see \cite[Proposition~9.1]{Li14}). We will only use (\ref{eqWeakConvPreImgToMWithWeight}) in this paper, namely, in the proof of Proposition~\ref{propPressureLimitPeriodicPts}.

\begin{prop}  \label{propWeakConvPreImgWithWeight}
Let $f$, $d$, $\phi$, $\alpha$ satisfy the Assumptions. Let $\mu_\phi$ be the unique equilibrium state for $f$ and $\phi$, let $m_\phi$ be the unique eigenmeasure of $\RR^*_\phi$, and $\widetilde\phi$ as defined in (\ref{eqDefPhiTilde}). For each sequence $\{ x_n \}_{n\in\N}$ of points in $S^2$, we define the Borel probability measures
\begin{align}
           \xi_n & = \frac{1}{Z_n(\phi)} \sum\limits_{y\in f^{-n}(x_n)} \deg_{f^n} (y) \exp\(S_n\phi(y)\) \delta_y,  \\
   \widehat\xi_n & = \frac{1}{Z_n(\phi)} \sum\limits_{y\in f^{-n}(x_n)} \deg_{f^n} (y) \exp\(S_n\phi(y)\) \frac{1}{n} \sum\limits_{i=0}^{n-1} \delta_{f^i(y)},\\
 \widetilde\xi_n & = \frac{1}{Z_n\big(\widetilde\phi \hspace{0.5mm} \big)} \sum\limits_{y\in f^{-n}(x_n)} \deg_{f^n} (y) \exp(S_n\widetilde\phi(y)) \delta_y, 
\end{align}
for each $n\in\N_0$, where $Z_n(\psi) = \sum\limits_{y\in f^{-n}(x_n)} \deg_{f^n} (y) \exp\(S_n\psi(y)\)$, for $\psi\in\CCC(S^2)$. Then
\begin{equation}  \label{eqWeakConvPreImgToMWithWeight}
\xi_n \stackrel{w^*}{\longrightarrow} m_\phi \text{ as } n\longrightarrow +\infty,
\end{equation}
\begin{equation}  \label{eqWeakConvPreImgToMuSumWithWeight}
\widehat\xi_n \stackrel{w^*}{\longrightarrow} \mu_\phi \text{ as } n\longrightarrow +\infty,
\end{equation}
\begin{equation}  \label{eqWeakConvPreImgToMuTildeWithWeight}
\widetilde\xi_n \stackrel{w^*}{\longrightarrow} \mu_\phi \text{ as } n\longrightarrow +\infty.
\end{equation}
\end{prop}

We now prove a slight generalization of Proposition~5.17 in \cite{Li14}.

\begin{prop}      \label{propPressureLimitPreImgs}
Let $f$, $d$, $\phi$, $\alpha$ satisfy the Assumptions. Then for each sequence $\{x_n\}_{n\in\N}$ in $S^2$, we have
\begin{equation}   \label{eqPressureLimitPreImgsWDeg}
P(f,\phi)   = \lim\limits_{n\to +\infty} \frac{1}{n} \log \sum\limits_{y\in f^{-n}(x_n)} \deg_{f^n}(y) \exp (S_n\phi(y)).
\end{equation}
If we also assume that $f$ has no periodic critical points, then for an arbitrary sequence of functions $\{w_n\: S^2\rightarrow \R\}_{n\in\N}$ satisfying $w_n(x)\in [1,\deg_{f^n}(x)]$ for each $n\in\N$ and each $x\in S^2$, we have
\begin{equation}  \label{eqPressureLimitPreImgsWODeg}
P(f,\phi)  = \lim\limits_{n\to +\infty} \frac{1}{n} \log \sum\limits_{y\in f^{-n}(x_n)} w_n(y) \exp (S_n\phi(y)). 
\end{equation}
\end{prop}

\begin{proof}
We fix a Jordan curve $\CC\subseteq S^2$ that satisfies the Assumptions (see \cite[Theorem~1.2]{BM10} or Lemma~\ref{lmCexistsL} for the existence of such $\CC$). By Proposition~5.17 in \cite{Li14}, for each $x\in S^2$ we have
$$
P(f,\phi)  = \lim\limits_{n\to +\infty} \frac{1}{n} \log \sum\limits_{y\in f^{-n}(x)} \deg_{f^n}(y) \exp (S_n\phi(y)).
$$
Combining this equation with (\ref{eqSigmaExpSnPhiBound}) in Lemma~\ref{lmSnPhiBound}, we get (\ref{eqPressureLimitPreImgsWDeg}).

Assume now that $f$ has no periodic critical points. Then there exists a finite number $M\in\N$ that depends only on $f$ such that $\deg_{f^n}(x) \leq M$ for $n\in\N_0$ and $x\in S^2$ \cite[Lemma~17.1]{BM10}. Thus for each $n\in\N$,
$$
1 \leq \frac{\sum\limits_{y\in f^{-n}(x_n)} \deg_{f^n}(y) \exp (S_n\phi(y))}{\sum\limits_{y\in f^{-n}(x_n)} w_n(y) \exp (S_n\phi(y))} \leq M.
$$
Hence (\ref{eqPressureLimitPreImgsWODeg}) follows from (\ref{eqPressureLimitPreImgsWDeg}).
\end{proof}

While Proposition~\ref{propPressureLimitPreImgs} is a statement for itereated preimages, the next proposition is for periodic points. Let $P_{1,f^n}=\{x\in S^2 \,|\, f^n(x)=x\}$ for $n\in\N$.

\begin{prop}      \label{propPressureLimitPeriodicPts}
Let $f$, $d$, $\phi$, $\alpha$ satisfy the Assumptions. Fix an arbitrary sequence of functions $\{w_n\: S^2\rightarrow \R\}_{n\in\N}$ satisfying $w_n(x)\in [1,\deg_{f^n}(x)]$ for each $n\in\N$ and each $x\in S^2$. Then
\begin{equation}  \label{eqPressureLimitPeriodicPts}
P(f,\phi) = \lim\limits_{n\to +\infty} \frac{1}{n} \log \sum\limits_{x\in P_{1,f^n}} w_n(x) \exp (S_n\phi(x)).
\end{equation}
\end{prop}

\begin{proof}
We fix a Jordan curve $\CC\subseteq S^2$ that satisfies the Assumptions (see \cite[Theorem~1.2]{BM10} or Lemma~\ref{lmCexistsL} for the existence of such $\CC$). 

We first prove (\ref{eqPressureLimitPeriodicPts}). By Proposition~\ref{propPressureLimitPreImgs}, it suffices to prove that there exist $C>1$ and $z\in S^2$ such that for each $n\in\N$ sufficiently large,
\begin{equation}   \label{eqSumsComparable}
\frac{1}{C}  \leq \frac{\sum\limits_{x\in P_{1,f^n}} w_n(x) \exp (S_n\phi(x))}{\sum\limits_{x\in f^{-n}(z)}\deg_{f^n}(x) \exp (S_n\phi(x))}  \leq C.
\end{equation}

We fix a $0$-edge $e_0 \subseteq \CC$ and a point $z\in \inte (e_0)$.

By Proposition~7.1 in \cite{Li14}, $m_\phi(\CC) = 0$. By the continuity of $m_\phi$, we can find $\delta >0$ such that
\begin{equation}
m_\phi \big( \overline{N^\delta_d} (\CC)\big) < \frac{1}{100}.
\end{equation}

Note that $\deg_{f^n}(y)=1$ if $f^n(y)=z$ for $n\in\N$. We define, for each $n\in\N_0$, the probability measure
\begin{align}  \label{eqDistrPreImgs}
\nu_n 
& = \sum\limits_{x\in f^{-n}(z)} \frac{\deg_{f^n} (x)\exp\(S_n\phi(x)\)}{\sum_{y\in f^{-n}(z)} \deg_{f^n} (y) \exp\(S_n\phi(y)\)}  \delta_x   \notag \\
& = \sum\limits_{x\in f^{-n}(z)} \frac{\exp\(S_n\phi(x)\)}{\sum_{y\in f^{-n}(z)} \exp\(S_n\phi(y)\)} \delta_x ,
\end{align}

Let $N_0\in \N$ be the constant from Lemma~\ref{lmAtMostOneCritPerTile}. By (\ref{eqWeakConvPreImgToMWithWeight}) in Proposition~\ref{propWeakConvPreImgWithWeight}, $\nu_n\stackrel{w^*}{\longrightarrow} m_\phi$ as $n\longrightarrow +\infty$. So by Lemma~\ref{lmPortmanteau}, we can choose $N_1>N_0$ such that for each $n\in\N$ with $n>N_1$, we have
\begin{equation}    \label{eqNbhdCFewPreImgs}
\nu_n \big( \overline{N^\delta_d} (\CC)\big) < \frac{1}{10}.
\end{equation}
By Lemma~\ref{lmCellBoundsBM}, it is clear that we can choose $N_2>N_1$ such that for each $n\in\N$ with $n>N_2$, and each $n$-tile $X^n\in\X^n$,
\begin{equation}   \label{eqDiamNTileSmall}
\diam_d(X^n) < \frac{\delta}{10}.
\end{equation}

We observe that for each $i\in\N$, we can pair a white $i$-tile $X_w^i\in\X_w^i$ and a black $i$-tile $X_b^i\in\X_b^i$ whose intersection $X_w^i\cap X_b^i$ is an $i$-edge contained in $f^{-i}(e_0)$.  There are a total of $d^i$ such pairs and each $i$-tile is in exactly one such pair. We denote by $\PP_i$ the collection of the unions $X_w^i\cup X_b^i$ of such pairs, i.e.,
$$
\PP_i=\{X_w^i\cup X_b^i  \,|\, X_w^i\in\X_w^i,X_b^i\in\X_b^i, X_w^i\cap X_b^i\cap f^{-i}(e_0) \in \E^i \}.
$$
We denote $\PP^\delta_i=\{A\in \PP_i \,|\, A \setminus N^\delta_d(\CC) \neq \emptyset\}$.

We now fix an integer $n> N_2$.

Then $\PP^\delta_n$ forms a cover of $S^2\setminus \overline{N^\delta_d}(\CC)$. For each $A\in\PP^\delta_n$, by (\ref{eqDiamNTileSmall}) we have $A\cap \CC =\emptyset$. So $A\subseteq \inte X^0_w$ or $A\subseteq  \inte X^0_b$, where $X^0_w$ and $X^0_b$ are the white $0$-tile and the black $0$-tile in $\X^0$, respectively. So by Brouwer's Fixed Point Theorem (see for example, \cite[Theorem~1.9]{Ha02}) and Lemma~\ref{lmAtMostOneCritPerTile}, we can define a function $p\:\PP^\delta_n\rightarrow P_{1,f^n}$ in such a way that $p(A)$ is the unique fixed point of $f^n$ contained in $A$. (For example, if $A\in\PP^\delta_n$ is the union of a black $n$-tile $X^n_b$ and a white $n$-tile $X^n_w$ and is a subset of the interior of the black $0$-tile, then there is no fixed point of $f^n$ in $X^n_w$, and by applying Brouwer's Fixed Point Theorem to the inverse of $f^n$ restricted to $X^n_b$, we get a fixed point $x\in X^n_b$ of $f^n$, which is the unique fixed point of $f^n$ in $X^n_b$ by Lemma~\ref{lmAtMostOneCritPerTile}.) Moreover, for each $A\in\PP^\delta_n$, $p(A)\in \inter A$, so $\deg_{f^n}(p(A)) =1 = w_n(p(A))$. In general, by Lemma~\ref{lmAtMostOneCritPerTile}, each $A\in\PP_n$ contains at most $2$ fixed points of $f^n$.

We also define a function $q\: \PP_n \rightarrow f^{-n}(z)$ in such a way that $q(A)$ is the unique preimage of $z$ under $f^n$ that is contained in $A$, for each $A\in\PP_n$ (see Proposition~\ref{propCellDecomp}). We note that if $X^n_w\in\X^n_w$ and $X^n_b\in\X^n_b$ are the $n$-tiles that satisfy $X^n_w\cup X^n_b =A\in\PP_n$ and $e_n=X^n_w\cap X^n_b$, then $q(A)\in e_n$. Thus in particular, $\deg_{f^n}(q(A))=1$ for each $A\in\PP_n$.

Hence by construction, we have
\begin{equation}    \label{eqSplitSumPreImgs}
\sum\limits_{x\in f^{-n}(z)} e^{S_n\phi(x)} = \sum\limits_{A\in \PP^\delta_n}  e^{S_n\phi(q(A))} + \sum\limits_{A\in \PP_n\setminus\PP^\delta_n}  e^{S_n\phi(q(A))},
\end{equation}
and
\begin{align}    \label{eqSplitSumPerPts}
     & \sum\limits_{A\in \PP^\delta_n} e^{S_n\phi(p(A))} \leq \sum\limits_{x\in P_{1,f^n}} w_n(x)   e^{S_n\phi(x)}\\
\leq & \sum\limits_{A\in \PP^\delta_n} e^{S_n\phi(p(A))} + \sum\limits_{A\in \PP_n\setminus\PP^\delta_n} \sum\limits_{x\in A\cap P_{1,f^n}}  e^{S_n\phi(x)}.   \notag
\end{align}
The last inequality in (\ref{eqSplitSumPerPts}) is due to the fact that if $x\in P_{1,f^n}$ satisfies $\deg_{f^n}(x) \geq 2$, then $x\in\V^n$ with $x\notin \bigcup \PP^\delta_n$, and the number of $A\in \PP_n$ that contains $x$ is at least $\deg_{f^n}(x)$ (and at most $2\deg_{f^n}(x)$).

By (\ref{eqSnPhiBound}) in Lemma~\ref{lmSnPhiBound}, we get
\begin{equation}    \label{eqQuotSumAwayC}
\frac{1}{C_3} \leq  \frac{\sum\limits_{A\in \PP^\delta_n}  e^{S_n\phi(p(A))}}{\sum\limits_{A\in \PP^\delta_n}   e^{S_n\phi(q(A))}} \leq C_3,
\end{equation}
and since in addition, $\card(A\cap P_{1,f^n}) \leq 2$ for $A\in\PP_n$ by Lemma~\ref{lmAtMostOneCritPerTile}, we have
\begin{equation}    \label{eqQuotSumNearC}
\frac{\sum\limits_{A\in \PP_n \setminus \PP^\delta_n} \sum\limits_{x\in A\cap P_{1,f^n}}  e^{S_n\phi(x)}}{\sum\limits_{A\in \PP_n \setminus \PP^\delta_n}    e^{S_n\phi(q(A))}} \leq 2 C_3,
\end{equation}
where 
$$
C_3 = \exp\(C_1 \(\diam_d(S^2)\)^\alpha\),
$$
and $C_1>0$ is a constant from Lemma~\ref{lmSnPhiBound}. Both $C_1$ and $C_3$ depend only on $f$, $\CC$, $d$, $\phi$, and $\alpha$.

By (\ref{eqSplitSumPreImgs}), (\ref{eqDistrPreImgs}), and (\ref{eqNbhdCFewPreImgs}), we get
\begin{equation}   \label{eqSumPreImgsComparable}
\sum\limits_{x\in f^{-n}(z)}  e^{S_n\phi(x)}  
\geq \sum\limits_{A\in \PP^\delta_n}  e^{S_n\phi(q(A))}
\geq \frac{9}{10} \sum\limits_{x\in f^{-n}(z)}   e^{S_n\phi(x)}.  
\end{equation}

Hence, by (\ref{eqSplitSumPerPts}), (\ref{eqQuotSumAwayC}), and (\ref{eqSumPreImgsComparable}), we have
\begin{equation*}
\frac{\sum\limits_{x\in P_{1,f^n}} w_n(x) e^{S_n\phi(x)}}{\sum\limits_{x\in f^{-n}(z)}\deg_{f^n}(x) e^{S_n\phi(x)}}
\geq \frac{\sum\limits_{A\in \PP^\delta_n} e^{S_n\phi(p(A))}}{\frac{10}{9} \sum\limits_{A\in \PP^\delta_n} e^{S_n\phi(q(A))}  }
\geq \frac{9}{10C_3}.
\end{equation*}
On the other hand, by (\ref{eqSplitSumPreImgs}), (\ref{eqSplitSumPerPts}), (\ref{eqQuotSumAwayC}), (\ref{eqQuotSumNearC}), and (\ref{eqSumPreImgsComparable}), we get
\begin{align*}
      & \frac{\sum\limits_{x\in P_{1,f^n}} w_n(x) e^{S_n\phi(x)}}{\sum\limits_{x\in f^{-n}(z)}\deg_{f^n}(x) e^{S_n\phi(x)}}   \leq \frac{ \sum\limits_{A\in \PP^\delta_n} e^{S_n\phi(p(A))} + \sum\limits_{A\in \PP_n\setminus\PP^\delta_n} \sum\limits_{x\in A\cap P_{1,f^n}}  e^{S_n\phi(x)} }{ \sum\limits_{x\in f^{-n}(z)}  e^{S_n\phi(x)} }  \\
 \leq & \frac{\sum\limits_{A\in \PP^\delta_n}  e^{S_n\phi(p(A))}}{\sum\limits_{A\in \PP^\delta_n}   e^{S_n\phi(q(A))}} +  \frac{ \sum\limits_{A\in \PP_n\setminus\PP^\delta_n} \sum\limits_{x\in A\cap P_{1,f^n}}  e^{S_n\phi(x)} }{ 10  \sum\limits_{A\in \PP_n\setminus\PP^\delta_n} e^{S_n\phi(q(A))} }     \leq C_3 + \frac{2}{10}C_3.
\end{align*}

Thus (\ref{eqSumsComparable}) holds if we choose $C=2 C_3$ and $n>N_2$. The proof is now complete.
\end{proof}

\subsection{Proof of large deviation principles}    \label{subsctProofLDP}

\begin{proof}[Proof of Theorem~\ref{thmLDP}]
Let $\phi \in \Holder{\alpha}(S^2,d)$ for some $\alpha\in(0,1]$.

We apply Theorem~\ref{thmAbsLargeDeviationPrincipleCRL} with $X=S^2$, $g=f$, and $H=\Holder{\alpha}(S^2,d)$. Note that $\Holder{\alpha}(S^2,d)$ is dense in $\CCC(S^2)$ with respect to the uniform norm (see for example, \cite[Lemma~6.12]{Li14}). Theorem~\ref{thmExistenceUniquenessES} implies Condition~(ii) in the hypothesis of Theorem~\ref{thmAbsLargeDeviationPrincipleCRL}. Condition~(i) follows from Corollary~\ref{corUSC}, (\ref{eqVPEntropy}), and the fact that $h_{\operatorname{top}}(f) = \log(\deg f)$ \cite[Corollary~20.8]{BM10}.  

It now suffices to verify (\ref{eqThmCCondition}) for each of the sequences $\{\Omega_n(x_n)\}_{n\in\N}$ and $\{\Omega_n\}_{n\in\N}$ of Borel probability measures on $\PPP(S^2)$.

Fix an arbitrary $\psi \in \Holder{\alpha}(S^2,d)$.

By (\ref{eqPressureLimitPreImgsWODeg}) in Proposition~\ref{propPressureLimitPreImgs},
\begin{align*}
  & \lim\limits_{n\to+\infty} \frac{1}{n} \log \int_{\PPP(S^2)} \! \exp \(n\int\!\psi \,\mathrm{d}\mu \) \,\mathrm{d}\Omega_n(x_n)(\mu) \\
= & \lim\limits_{n\to+\infty} \frac{1}{n} \log \sum\limits_{y\in f^{-n}(x_n)}  \frac{ w_n(y) \exp(S_n\phi(y))  }{\sum_{z\in f^{-n}(x_n)} w_n(z) \exp(S_n\phi(z))}  e^{\sum_{i=0}^{n-1} \psi\(f^i(y)\) } \\
= & \lim\limits_{n\to+\infty} \frac{1}{n} \Bigg( \log \sum\limits_{y\in f^{-n}(x_n)}   w_n(y) e^{S_n(\phi+\psi)(y)} - \log \sum\limits_{z\in f^{-n}(x_n)}   w_n(z) e^{S_n(\phi)(z)} \Bigg) \\
= & P(f,\phi+\psi) - P(f,\phi).
\end{align*}

Similarly, by (\ref{eqPressureLimitPeriodicPts}) in Proposition~\ref{propPressureLimitPeriodicPts}, we get
$$
   P(f,\phi+\psi) - P(f,\phi)
=  \lim\limits_{n\to+\infty} \frac{1}{n} \log \int_{\PPP(S^2)} \! \exp \(n\int\!\psi \,\mathrm{d}\mu \) \,\mathrm{d}\Omega_n(\mu)
$$
The theorem now follows from Theorem~\ref{thmAbsLargeDeviationPrincipleCRL}.
\end{proof}

\subsection{Equidistribution with respect to the equilibrium state}  \label{subsctEquidistr}

\begin{proof}[Proof of Corollary~\ref{corMeasTheoPressure}]
We prove the first equality in (\ref{eqMeasTheoPressure}) now.

Fix $\mu \in \MMM(S^2,f)$ and a convex local basis $\GGG_\mu$ at $\mu$. By (\ref{eqRateFnI}) and the upper semi-continuity of $h_\mu(f)$ (Corollary~\ref{corUSC}), we get 
$$
-I^\phi(\mu) = \inf\limits_{\GG\in\GGG_\mu}\bigg(\sup\limits_\GG (-I^\phi)\bigg) = \inf\limits_{\GG\in\GGG_\mu}\(-\inf\limits_\GG I^\phi\).
$$
Then by (\ref{eqRateFnI}) and (\ref{eqInfI}),
\begin{align*}
 & -P(f,\phi) + \int\!\phi\,\mathrm{d}\mu + h_\mu(f) = -I^\phi(\mu) = \inf\limits_{\GG\in\GGG_\mu}\(-\inf\limits_\GG I^\phi\)   \\
=& \inf\limits_{\GG\in\GGG_\mu} \Bigg\{ \lim\limits_{n\to+\infty} \frac{1}{n}\log \sum\limits_{y\in f^{-n}(x_n),\, W_n(y)\in\GG} \frac{w_n(y)\exp(S_n\phi(y))}{Z_n(\phi)} \Bigg\},
\end{align*}
where we write $Z_n(\phi)= \sum\limits_{z\in f^{-n}(x_n)} w_n(z) \exp (S_n\phi(z))$. By (\ref{eqPressureLimitPreImgsWODeg}) in Proposition~\ref{propPressureLimitPreImgs}, we have $P(f,\phi)=\lim\limits_{n\to+\infty} \frac{1}{n} \log Z_n(\phi)$. Thus the first equality in (\ref{eqMeasTheoPressure}) follows.

\smallskip

By similar arguments, with (\ref{eqPressureLimitPreImgsWDeg}) in Proposition~\ref{propPressureLimitPreImgs} replaced by  (\ref{eqPressureLimitPeriodicPts}) in Proposition~\ref{propPressureLimitPeriodicPts}, we get the second equality in (\ref{eqMeasTheoPressure}).
\end{proof}

\begin{proof}[Proof of Corollary~\ref{corEquidistr}]
Recall that $W_n(x)=\frac{1}{n}\sum\limits_{i=0}^{n-1} \delta_{f^i(x)} \in \PPP(S^2)$ for $x\in S^2$ and $n\in\N$ as defined in (\ref{eqDefWn}). We write 
$$
Z^+_n(\GG)=\sum\limits_{y\in f^{-n}(x_n), \, W_n(y)\in\GG}  \deg_{f^n}(y)\exp(S_n\phi(y))
$$
and 
$$
Z^-_n(\GG)=\sum\limits_{y\in f^{-n}(x_n), \, W_n(y)\notin\GG}  \deg_{f^n}(y)\exp(S_n\phi(y))
$$
for each $n\in\N$ and each open set $\GG\subseteq \PPP(S^2)$.

Let $\GGG_{\mu_\phi}$ be a convex local basis of $\PPP(S^2)$ at $\mu_\phi$. Fix an arbitrary convex open set $\GG\in\GGG_{\mu_\phi}$. 

By the uniqueness of the equilibrium state in our context and Corollary~\ref{corMeasTheoPressure}, we get that for each $\mu\in\PPP(S^2)$, there exist numbers $a_\mu < P(f,\phi)$ and $N_\mu \in\N$ and an open neighborhood $\UU_\mu \in \PPP(S^2)\setminus \{\mu_\phi\}$ containing $\mu$ such that for each $n>N_\mu$,
\begin{equation}  \label{eqZn<}
Z_n^+(\UU_\mu) \leq \exp(n a_\mu).
\end{equation}
Since $\PPP(S^2)$ is compact in the weak$^*$ topology by Alaoglu's theorem, so is $\PPP(S^2)\setminus \GG$. Thus there exists a finite set $\{\mu_i\,|\, i\in I\} \subseteq \PPP(S^2)\setminus \GG$ such that
\begin{equation}   \label{eqUcover}
\PPP(S^2)\setminus \GG \subseteq \bigcup\limits_{i\in I} \UU_{\mu_i}.
\end{equation}
Here $I$ is a finite index set. Let $a=\max\{ a_{\mu_i} \,|\, i\in I\}$. Note that $a<P(f,\phi)$. By Corollary~\ref{corMeasTheoPressure} with $\mu=\mu_\phi$, we get that
\begin{equation}  \label{eqPfcorEquidistr1}
P(f,\phi_0) \leq \lim\limits_{n\to+\infty} \frac{1}{n} \log Z^+_n(\GG).
\end{equation}
Combining (\ref{eqPfcorEquidistr1}) with (\ref{eqPressureLimitPreImgsWDeg}) in Proposition~\ref{propPressureLimitPreImgs}, we get that the equality holds in (\ref{eqPfcorEquidistr1}). So there exist numbers $b\in (a,P(f,\phi))$ and $N\geq \max\{ N_i \,|\, i\in I\}$ such that for each $n> N$,
\begin{equation} \label{eqZn>}
Z_n^+(\GG) \geq \exp(nb).
\end{equation}

We claim that every subsequential limit of $\{\nu_n\}_{n\in\N}$ in the weak$^*$ topology lies in the closure $\overline{\GG}$ of $\GG$. Assuming that the claim holds, then since $\GG\in\GGG_{\mu_\phi}$ is arbitrary, we get that any subsequential limit of $\{\nu_n\}_{n\in\N}$ in the weak$^*$ topology is $\mu_\phi$, i.e., $\nu_n \stackrel{w^*}{\longrightarrow} \mu_\phi$ as $n\longrightarrow +\infty$.

\smallskip

We now prove the claim. We first observe that for each $n\in\N$,
\begin{align*}
\nu_n & = \sum\limits_{y\in f^{-n}(x_n)} \frac{w_n(y) \exp(S_n\phi(y))}{ Z^+_n(\GG) + Z^-_n(\GG)} W_n(y) \\
      & = \frac{Z^+_n(\GG)}{Z^+_n(\GG) + Z^-_n(\GG)} \nu'_n + \sum\limits_{y\in f^{-n}(x_n),\,W_n(y)\notin\GG} \frac{w_n(y) e^{S_n\phi(y)}}{ Z^+_n(\GG) + Z^-_n(\GG)} W_n(y),
\end{align*}
where $\nu'_n = \sum\limits_{y\in f^{-n}(x_n),\,W_n(y)\in\GG} \frac{w_n(y) \exp(S_n\phi(y))}{ Z^+_n(\GG)} W_n(y)$.

Note that since $a<b$, by (\ref{eqUcover}), (\ref{eqZn<}), and (\ref{eqZn>}),
$$
0\leq \lim\limits_{n\to+\infty} \frac{Z_n^-(\GG)}{Z_n^+(\GG)} \leq \lim\limits_{n\to+\infty} \frac{\sum_{i\in I} Z_n^+(\UU_i)}{Z_n^+(\GG)} \leq \lim\limits_{n\to+\infty} \frac{\card(I)\exp(na)}{\exp(nb)} = 0. 
$$
So $\lim\limits_{n\to+\infty} \frac{Z^+_n(\GG)}{Z^+_n(\GG) + Z^-_n(\GG)}  = 1$, and that the total variation
\begin{align*}
     & \Bigg\|{\sum\limits_{y\in f^{-n}(x_n),\,W_n(y)\notin\GG} \frac{w_n(y) \exp(S_n\phi(y))}{ Z^+_n(\GG) + Z^-_n(\GG)} W_n(y)} \Bigg\|   \\
     \leq & \frac{\sum\limits_{y\in f^{-n}(x_n),\,W_n(y)\notin\GG} w_n(y) \exp(S_n\phi(y)) \Norm{W_n(y)}}{Z^+_n(\GG) + Z^-_n(\GG)} \\
     \leq & \frac{Z^-_n(\GG)}{Z^+_n(\GG) + Z^-_n(\GG)}  \longrightarrow 0
\end{align*}
as $n\longrightarrow +\infty$. Thus a measure is a subsequential limit of $\{\nu_n\}_{n\in\N}$ if and only if it is a subsequential limit of $\{\nu'_n\}_{n\in\N}$. Note that $v'_n$ is a convex combination of measures in $\GG$, and $\GG$ is convex, so $\nu'_n\in\GG$, for $n\in\N$. Hence each subsequential limit of $\{\nu_n\}_{n\in\N}$ lies in the closure $\overline{\GG}$ of $\GG$. The proof of the claim is complete now.

\smallskip

By similar arguments as in the proof of the convergence of $\{\nu_n\}_{n\in\N}$ above, with (\ref{eqPressureLimitPreImgsWODeg}) in Proposition~\ref{propPressureLimitPreImgs} replaced by (\ref{eqPressureLimitPeriodicPts}) in Proposition~\ref{propPressureLimitPeriodicPts}, we get that $\eta_n \stackrel{w^*}{\longrightarrow} \mu_\phi$ as $n\longrightarrow +\infty$.
\end{proof}

\end{document}